\documentclass[12pt]{amsart}
\usepackage[active]{srcltx}
\usepackage{latexsym}
\usepackage{hyperref}

\usepackage{amsmath}
\newcommand{\ignore}[1]{}
\newcommand{\hide}[1]{}
\DeclareMathOperator{\ad}{ad}

\newcommand{\F}{\mathbb F}

\newtheorem{dummy}{Dummy}

\newtheorem{lemma}[dummy]{Lemma}

\newtheorem{theorem}[dummy]{Theorem}

\newtheorem{prop}[dummy]{Proposition}
\newtheorem{cor}[dummy]{Corollary}

\theoremstyle{definition}
\newtheorem{definition}[dummy]{Definition}

\newtheorem{convention}[dummy]{Convention}

\theoremstyle{remark}
\newtheorem{rem}[dummy]{Remark}
\newtheorem*{rem*}{Remark to ourselves}
\begin{document}%_______________________________________

\bibliographystyle{amsalpha}

\author{Marina Avitabile}
\email{marina.avitabile@unimib.it}
\address{Dipartimento di Matematica e Applicazioni\\
  Universit\`a degli Studi di Milano - Bicocca\\
 via Cozzi 55\\
  I-20125 Milano\\
  Italy}
\author{Sandro Mattarei}
\email{smattarei@lincoln.ac.uk}
\address{Charlotte Scott Centre for Algebra\\
University of Lincoln \\
Brayford Pool
Lincoln, LN6 7TS\\
United Kingdom}

% \title[Nottingham algebras with diamonds of infinite type]{On the structure of Nottingham algebras with diamonds of infinite type}

\title[The earliest diamond of finite type]{The earliest diamond of finite type in Nottingham algebras}

\subjclass[2020]{Primary 17B50; secondary 17B70, 17B65}
%    The 2010 edition of the Mathematics Subject Classification is
%    now available.  If you are citing a classification from the
%    new scheme, use the following input coding instead.
\keywords{Modular Lie algebra, graded Lie algebra, thin Lie algebra}

\begin{abstract}
We prove several structural results on {\em Nottingham algebras,} a class of infinite-dimensional, modular, graded Lie algebras,
which includes the graded Lie algebra associated to the Nottingham group with respect to its lower central series.
Homogeneous components of a Nottingham algebra have dimension one or two, and in the latter case they are called {\em diamonds}.
The first diamond occurs in degree $1$, and the second occurs in degree $q$, a power of the characteristic.
Each diamond past the second is assigned a {\em type,} which either belongs to the underlying field or is $\infty$.

%We prove that the difference in degrees of any two consecutive diamonds in a Nottingham algebra equals $q-1$.
%We classify Nottingham algebras where all diamonds have type $\infty$.
Nottingham algebras with a variety of diamond patterns are known.
In particular, some have diamonds of both finite and infinite type.
We prove that each of those known examples is uniquely determined by a certain finite-dimensional quotient.
Finally, we determine how many diamonds of type $\infty$ may precede the earliest diamond of finite type in an arbitrary Nottingham algebra.
\end{abstract}

\maketitle
\section{Introduction}\label{sec:intro}

% A {\em thin} Lie algebra is a graded Lie algebra $L=\bigoplus_{i=1}^{\infty}L_i$
% with $\dim L_1=2$, and such that
% any nonzero graded ideal $I$ of a thin Lie algebra $L$ is constrained
% between two consecutive Lie powers of $L$, in the sense that $L^i\supseteq I\supseteq L^{i+1}$ for some $i$.

A {\em thin} Lie algebra is a graded Lie algebra $L=\bigoplus_{i=1}^{\infty}L_i$
with $\dim L_1=2$ and satisfying the following {\em covering property:}
for each $i$, each nonzero $z\in L_i$ one has $[zL_1]=L_{i+1}$.
(Note that we write Lie products without a comma.)
This implies at once that homogeneous components of a thin Lie algebra are at most two-dimensional.
Those components of dimension two are called {\em diamonds,}
hence $L_1$ is a diamond, and if there are no other diamonds then $L$ is a {\em graded Lie algebra of maximal class}~\cite{CMN,CN}.
It is convenient, however, to explicitly exclude graded Lie algebras of maximal class from the definition of thin Lie algebras.
Thus, a thin Lie algebra must have at least one further diamond besides $L_1$ (which we may call the {\em first} diamond of $L$),
and we let $L_k$ be the earliest such diamond (the {\em second} diamond).
%(although the presence of one-dimensional {\em fake diamonds} at certain locations)
For the sake of simplicity in this introduction we assume all thin Lie algebras to have infinite dimension.

% The term {\em diamond} originates from a lattice-theoretic characterization of thin Lie algebras motivated by~\cite{Br}.
% In fact, any nonzero graded ideal $I$ of a thin Lie algebra $L$ is constrained
% between two consecutive Lie powers of $L$, in the sense that $L^i\supseteq I\supseteq L^{i+1}$ for some $i$,
% and hence the lattice of graded ideals looks like a sequence of diamonds
% (a name for the lattice of subspaces of a two-dimensional vector space) connected by {\em chains}.
% We will not formally assign numerical lengths to those chains but it should be clear that they are important in describing the structure of a thin Lie algebra.
% Knowing those lengths amounts to knowing the degrees in which the diamonds occur.

The most basic invariant of a thin Lie algebra is the degree $k$ of the second diamond.
It is known from~\cite{CMNS} and~\cite{AviJur}
that if the characteristic $p$ is different from $2$ then $k$ can
only be one of $3$, $5$, $q$, or $2q-1$,
where $q$ is a power of $p$.
% An alternate treatment of those arguments is given
% in~\cite{Mat:chain_lengths}.
% In particular, only $3$ and $5$ can occur in characteristic zero.
% In fact, in characteristic zero or larger than $5$, thin Lie algebras with $k=3$ or $5$,
% with the further assumption $\dim(L_4)=1$ in the former case,
% were shown in~\cite{CMNS} to belong to up to three isomorphism types,
% associated to $p$-adic Lie groups of types $A_1$ and $A_2$.
% (Explicit matrix realizations were given
% in~\cite{Mat:thin-groups}.)
% The values $q$ and $2q-1$ occur for two wide classes of thin Lie algebras
% built from certain nonclassical finite-dimensional
% simple modular Lie algebras, and also for thin Lie algebras obtained from graded Lie algebras of maximal class through various constructions.

In this paper we focus on thin Lie algebras with second diamond $L_q$.
One remarkable example of such thin Lie algebras arises as the graded Lie algebra associated to the lower central series
of the {\em Nottingham group} over the prime field $\F_p$, for $p$ odd~\cite{Car:Nottingham}.
That algebra has its second diamond in degree $p$, but admits a natural generalization with a power $q$ of $p$ in place of $p$.
For this reason thin Lie algebras with second diamond $L_q$ have been called {\em Nottingham algebras} in the literature.
However, because of exceptional behaviour in small characteristics here we reserve the name {\em Nottingham algebras} to thin Lie algebras
of characteristic $p>3$, having second diamond $L_q$ with $q>5$.

A wide variety of Nottingham algebras are known.
Several arise from certain cyclic gradings of various simple Lie algebras of Cartan type.
In particular, the thin Lie algebra associated with the Nottingham group arises from a cyclic grading of the {\em Witt algebra}.
Further Nottingham algebras, and in fact uncountably many ones, are closely related to graded Lie algebras of maximal class.
We refer the reader to Theorem~\ref{thm:k=q} and the discussion which follows it for a comprehensive survey.

In an arbitrary Nottingham algebra
each diamond past the first can be attached a {\em type},
which is an element of the underlying field, or  $\infty$.
% and describes certain relations hold.
The second diamond $L_q$ has invariably type $-1$, and we assign no type to the first diamond $L_1$.
The type of a diamond $L_m$ describes the adjoint action of $L_1$ on $L_m$, in such a way that knowledge of
all degrees in which diamonds occur in $L$, and their types, determines $L$ up to isomorphism.
It is necessary to include {\em fake diamonds} in such a description.
Those are in fact one-dimensional components, as we explain in Section~\ref{sec:types}, and correspond to types $0$ and $1$.
We refer to diamonds which are not fake, and thus are two-dimensional, as {\em genuine} diamonds.
Furthermore, the difference in degree of any two
consecutive diamonds equals $q-1$, see Section~\ref{sec:main}.

Various patterns of diamond types occur in Nottingham algebras.
One possible pattern has all diamonds of infinite type, with the necessary exception of the first two.
Such algebras were constructed in~\cite{Young:thesis},
they have second diamond $L_q$, of type $-1$,
and diamonds of infinite type in each higher degree
congruent to $1$ modulo $q-1$. A complementary uniqueness result is proved in~\cite{AviMat:diamond_distances}, where
is showed that there exists a unique Nottingham algebra
with second diamond in degree $q$ and all other diamonds having infinite type.

Another possible pattern sees all diamond types follow an arithmetic progression in the underlying field.
A special case of that arises for the algebra associated with the Nottingham group, where the sequence of types is constant.
More generally, there are Nottingham algebras having bunches of diamonds of infinite type, separated by single occurrences of diamonds of finite type,
where the finite types again follow an arithmetic progression.
We give a more detailed description in Theorem~\ref{thm:k=q},
which is an existence result for Nottingham algebras with those diamond patterns
and summarizes conclusions of several papers~\cite{Car:Nottingham,Avi,AviMat:A-Z,AviMat:mixed_types}.
Existence of the Nottingham algebras with all diamonds, but the second one, of infinite type also follows from~\cite{AviMat:A-Z}
as a limit case.

Uniqueness results for the algebras of Theorem~\ref{thm:k=q} where all diamonds have finite types were established in~\cite{CaMa:Nottingham},
in the sense that each of them is uniquely determined by an appropriate finite-dimensional quotient.
This was done by providing finite presentations for those algebras, or some central (or second central) extensions of them.
One of our goals is extending this to a uniqueness proof for the remaining algebras of Theorem~\ref{thm:k=q},
namely those with diamonds of both finite and infinite types.

\begin{theorem}\label{thm:presentation}
There is a unique infinite-dimensional Nottingham algebra $L$ with second diamond in degree $q$,
%a power of $p>3$,
with diamonds of infinite type in all
degrees $k(q-1)+1$ for $1<k\leq p^s$, where $s\geq 1$, and a diamond of finite type
$\lambda \neq 0$ in degree
$(p^s+1)(q-1)+1$.
Furthermore, there is a central extension of $L$ which is finitely presented.
\end{theorem}

In Section~\ref{sec:fin_pres} we prove Theorem~\ref{thm:fin_pres}, which is
a more precise version of Theorem~\ref{thm:presentation}.
We briefly discuss the exceptional case $\lambda=0$ in Remark~\ref{rem:lambda=0}.

We give a further contribution to the classification
project for Nottingham algebras
with the following result,
which determines the possible degree of the earliest diamond of finite type
of a Nottingham algebra after a first bunch of diamonds of infinite type.

\begin{theorem}\label{thm:main_intro}
%Let $L=\bigoplus_{i=1}^{\infty}L_{i}$
Let $L$ be an infinite-dimensional Nottingham algebra
with second diamond in degree $q$.
Suppose $L_{2q-1}$ is a diamond of infinite type,
%and that the centralizer in $L_1$ of every homogenous component from $L_{q+1}$ up to $L_{2q-3}$ equals the centralizer of $L_2$.
and suppose that the next diamond of finite type, that is, the earliest diamond of finite type past $L_{q}$,
occurs in degree $a(q-1)+1$, and has type $\mu\in \F$.
%Let $a\geq 3$ be an integer. Suppose that $L$
%has diamonds of infinite type in each degree $k(q-1)+1$, $2\leq k \leq a-1$,
%and a (possibly fake) diamond of finite type $\mu \in \F$ in degree $a(q-1)+1$.
Then either $a-1$ equals a power of $p$, or $a$ equals twice a power of $p$ and $\mu=1$.
\end{theorem}

Note that according to the fact that the difference in degree of any two consecutive diamonds equals $q-1$,
%Theorem~\ref{thm:distance_intro}
the degree of the earliest diamond of finite type past $L_{q}$ must have the form $a(q-1)+1$
for some integer $a$.
Theorem~\ref{thm:main_intro} will follow at once from a more precise and technical analogue where $L$ is allowed to be finite-dimensional, Theorem~\ref{thm:main}.
Of the two alternate conclusions on $a$ stated in Theorem~\ref{thm:main_intro},
the former occurs for the Nottingham algebras of Theorem~\ref{thm:presentation},
and the latter occurs for some of the Nottingham algebras studied by David Young in his PhD thesis~\cite{Young:thesis}.
We discuss some of those briefly at the end of Section~\ref{sec:types}
We now explain how our results complete
a piece of a classification of Nottingham algebras.
Thus, consider an infinite-dimensional Nottingham algebra $L$, with second diamond $L_{q}$,
and make the additional assumption $p>5$ (this restriction being inherited from~\cite{CaMa:Nottingham}).
%Suppose that $L$ has at least one further genuine diamond of finite type past $L_{q}$, and let $L_{m}$ be the earliest.
Suppose that $L$ has at least one further diamond of finite type past $L_q$,
let $L_{m}$ be the earliest,
and assume that $L_{m}$ is genuine.

Then either $m=2q-1$, meaning there are no diamonds of infinite type before $L_m$,
or $m=(p^s+1)(q-1)+1$ according to Theorem~\ref{thm:main_intro}, for some $s\geq 1$.
In the former case $L$ is isomorphic to one of the algebras described
in~\cite{Car:Nottingham} and~\cite[Theorems~2.3 and~2.4]{CaMa:Nottingham}.
In particular, $L$ has diamonds of finite type in each degree congruent to $1$ modulo $q-1$,
with the diamond types following an arithmetic progression
(including the fake ones).
%Furthermore, $L$ itself or a central (or second central) extension of $L$ is finitely presented.
In the latter case $L$ is isomorphic to one of the algebras described in
Theorem~\ref{thm:presentation}.
In particular, according to its full description recalled in Theorem~\ref{thm:k=q},
the algebra $L$ has diamonds in all degrees
of the form $t(q-1)+1$.
Those diamonds have infinite type except for
$t\equiv 1 \pmod {p^s}$, where the diamonds have
finite types following an arithmetic progression.
Allowing the parameter $s$ to be zero we obtain the following uniform description of the structure of the
Nottingham algebras under consideration.

\begin{theorem}\label{thm:3}
Let $L$ be an infinite-dimensional Nottingham algebra in characteristic $p>5$,
with second diamond in degree $q$.
Suppose $L$ has at least one genuine diamond of finite type past $L_{q}$,
let $L_{m}$ be the earliest,
and assume that $L_{m}$ has type $\lambda\neq 0,1$.
Then $m=(p^s+1)(q-1)+1$ for some $s\geq 0$.

Furthermore, $L$ has diamonds in all degrees congruent to $1$ modulo $q-1$.
The diamond $L_{t(q-1)+1}$ has finite type (possibly fake) if $t\equiv 1 \pmod {p^s}$, and infinite type otherwise.
 %If $t\not \equiv 1 \pmod {p^s}$ the corresponding diamond is of infinite type.
 %If $t\equiv 1 \pmod {p^s}$, the corresponding diamond has finite type.
The diamonds of finite type (which include fake diamonds if $\lambda \in \F_{p}$) follow an arithmetic progression.
Furthermore, $L$ itself (if $\lambda \not \in \F_{p}$ and $s=0$) or a central (if $\lambda \in \F_{p}$,
$\lambda \neq -2$ and $s=0$ or $\lambda \in \F$ and $s\geq 1$) or second central (if $\lambda =-2$) extension of $L$ is finitely presented.
\end{theorem}

%\newpage
\section{Nottingham algebras}\label{sec:types}

Recall from the Introduction that a {\em thin} Lie algebra is a graded Lie algebra $L=\bigoplus_{i=1}^{\infty}L_i$,
with $\dim L_1=2$ and satisfying the {\em covering property:}
for each $i$, each nonzero $z\in L_i$ satisfies $[zL_1]=L_{i+1}$.
In particular, if $L$ has infinite dimension then its centre is trivial.
If $L$ has finite dimension then its centre coincides with its highest nonzero component.
For convenience in Section~\ref{sec:intro}
we summarized our results for $L$ of infinite dimension,
but our Theorem~\ref{thm:main}, which is a stronger and more precise version of Theorem~\ref{thm:main_intro},
provides sharp information on certain finite-dimensional thin Lie algebras $L$.
% Consequently, we cannot take infinite-dimensionality as a simplifying blanket assumption in this paper.
% but instead we will implicitly assume that $L$ has {\em large enough dimension} in each of our definitions or results.
% This should be understood as enough to exclude
% the vanishing of certain homogeneous components of $L$,
% either explicitly mentioned of clear from the context.

As we mentioned in the Introduction, our definition of a Nottingham algebra includes restrictions on $p$ and $q$,
which we will partly justify below.

\begin{definition}
In this paper a {\em Nottingham algebra} is a thin Lie algebra, over a field of characteristic $p>3$,
with second diamond $L_q$, where $q>5$ is a power of $p$.
\end{definition}

In this paper we use the left-normed convention
for iterated Lie products, hence $[abc]$ stands for $[[ab]c]$.
We also use the shorthand
$[ab^i]=[ab\cdots b]$,
where $b$ occurs $i$ times.

Let $L$ be a Nottingham algebra with second diamond $L_q$.
We set up some standard notation.
There is a nonzero element $y$ of $L_1$, unique up to a scalar multiple, such that $[L_2y]=0$.
Extending to a basis $x,y$ of $L_1$, this means $[yxy]=0$.
This choice of $y$ implies $[Lyy]=0$, which means $(\ad y)^2=0$,
see~\cite{Mat:sandwich} for a more general result.
It is not hard to deduce from this relation that no two consecutive components of $L$
can both be diamonds, see~\cite{Mat:sandwich} for a proof.
% Thus, any two diamonds are separated by one or more one-dimensional homogeneous components.

The element $y$ centralizes each homogeneous component from $L_{2}$ up to $L_{q-2}$.
That is an nontrivial assertion proved in~\cite{CaJu:quotients}, and relies on the theory of graded Lie algebras of maximal class established in~\cite{CMN,CN}.
Consequently, $L_i$ is spanned by $[yx^{i-1}]$ for $2\le i<q$.
In particular, $v_1=[yx^{q-2}]$ spans the component $L_{q-1}$ and, in turn, $[v_1x]$ and $[v_1y]$ span the second diamond $L_q$.
(The meaning of the subscript in $v_1$ will be revealed in Section~\ref{sec:main}.)
It is now easy to see that one may redefine $x$ in such a way that
\[
[v_1xx]=0=[v_1yy] \quad\textrm{and} \quad [v_1yx]=-2[v_1xy],
\]
see~\cite[Section~3]{AviMat:A-Z} for a cleaner excerpt of the original argument in~\cite{Car:Nottingham}.
 %but once that is achieved it is easy to deduce
 %(see~\cite[Section~3]{AviMat:A-Z} for a clearer excerpt of the original argument in~\cite{Car:Nottingham})
 %that one can choose
 %$x\in L_{1}\setminus \F y$
 %satisfying
 %\[
 %[v_1xx]=0=[v_1yy] \quad \textrm{and} \quad [v_1yx]=-2[v_1xy],
 %\]
 %where $v_1$ is any nonzero element in degree $q-1$.
In the rest of this paper we refer to such $x$ and $y$ as {\em standard generators} of $L$.
Each of them is only determined up to a scalar multiple, but a different choice will not affect our definitions below,
in particular the definition of a diamond's type.
Because $[yx^q]=[v_1xx]=0$, we have $(\ad x)^q=0$.
Indeed, since $(\ad x)^q$ is a derivation of $L$, its kernel is a subalgebra, but then that must equal $L$ as both generators $x$ and $y$ belong to it.

We recall the definition of {\em type} of a diamond as introduced in ~\cite{CaMa:Nottingham}.
(Note that diamond types are defined differently for thin Lie algebras with second diamond $L_{2q-1}$, see~\cite{CaMa:thin}.)
We do not assign a type to the first diamond $L_1$.
Let then $L_m$ be a diamond past $L_1$, that is, a two-dimensional homogeneous component of $L$ with $m>1$,
and assume $L_{m+1}\neq 0$ to avoid trivial cases.
Because no two consecutive homogeneous components can be diamonds, $L_{m-1}$ is one-dimensional, and so is $L_{m+1}$.
If $w$ spans $L_{m-1}$, then $L_m$ is spanned by $[wx]$ and $[wy]$,
and $L_{m+1}$ is spanned by $[wxx]$, $[wxy]$, $[wyx]$ and $[wyy]$.
The following definition encodes particular relations
among these four elements.

 %A result in~\cite{Mat:sandwich} ensures that $L_{m-1}$ is one-dimensional.
 %(That holds for any thin Lie algebra, not just Nottingham algebras, under the sole assumption $\dim(L_{3})=1$.)

\begin{definition}\label{def:type}
Let $L$ be a Nottingham algebra,
with second diamond $L_q$ and standard generators $x$ and $y$.
%Any two-dimensional component $L_m$ of $L$ is called a {\em diamond}.
Let $L_m$ be a diamond of $L$, with $m>1$, and assume $L_{m+1}\neq 0$.
Let $w$ be a nonzero element in $L_{m-1}$.
%thus $L_{m}$ is spanned by the elements $[wx]$ and $[wy]$.
\begin{itemize}
\item[(a)]
We say $L_m$ is a diamond of {\em finite type} $\mu$, where $\mu\in\F$,
if
\begin{equation*}
[wxx]=0=[wyy] \quad\textrm{and} \quad \mu[wyx]=(1-\mu)[wxy].
\end{equation*}
\item[(b)]
We say $L_m$ is a diamond of {\em infinite type} if
\begin{equation*}
[wxx]=0=[wyy] \quad\textrm{and} \quad [wyx]=-[wxy].
\end{equation*}
%or {\em a diamond with a type} when we wish to leave the type unspecified.
\end{itemize}
\end{definition}

In particular, this definition applies to the second diamond $L_{q}$, which therefore has invariably type $\mu=-1$.
It is shown in~\cite{AviMat:diamond_distances} that every
diamond of $L$ satisfies Definition~\ref{def:type}
for some $\mu\in\F\cup\{\infty\}$.

% \hrule\fbox{Forse in questo paper non serve. Spiegare comunque genuino.}
% We will call a diamond $L_m$ {\em a diamond with a type} if $[wxx]=0$, whence those relations hold for some value of $\mu$ which we wish to leave unspecified.
% This terminology will include the special cases $\mu=0$ and $\mu=1$, which we discuss next.
% In Theorem~\ref{cor:next_diamond} and Theorem~\ref{thm:chain_recap} we establish that diamonds satisfy Definition~\ref{def:type} in great generality.
% \hrule

The values $\mu=0$ and $\mu=1$ cannot actually occur in
Definition~\ref{def:type}.
If $\mu=0$ then the relations $[wxx]=0=[wxy]$
would imply that the element $[wx]$ is central.
Similarly, if $\mu=1$ then the element $[wy]$ would be central.
%, since $[wyx]=0=[wyy]$.
However, because of the covering property and because $L_{m+1}\neq 0$,
no nonzero element of $L_m$ can be central.
Hence $[wx]=0$ if $\mu=0$, and $[wy]=0$ if $\mu=1$,
contradicting the two-dimensionality of $L_m$.

Thus, strictly speaking, diamonds of type $0$ or $1$ cannot occur,
at least if we insist that a diamond should have dimension two,
as in Definition~\ref{def:type}.
Nevertheless, it is convenient for a uniform description
of the diamonds patterns
in Nottingham algebras to allow ourselves to
call {\em diamonds of type $0$ or $1$}
certain one-dimensional homogeneous components $L_m$,
as long as they satisfy the relations of
Definition~\ref{def:type}
with $\mu=0$ or $1$.
This leads us to the following definition.

\begin{definition}\label{def:type-fake}
Let $L$ be a Nottingham algebra,
with second diamond $L_q$ and standard generators $x$ and $y$.
Let $L_{m-1}$ be a one-dimensional component, spanned by $w$, with $m>1$,
and assume $L_{m+1}\neq 0$.
\begin{itemize}
\item[(a)]
We say $L_m$ is a diamond of {\em of type $1$} if
\begin{equation*}
[wxx]=0 \quad\textrm{and} \quad [wy]=0.
\end{equation*}
\item[(b)]
We say $L_m$ is a diamond {\em of type $0$} if
\begin{equation*}
[wx]=0 \quad\textrm{and} \quad [wyy]=0.
\end{equation*}
%or {\em a diamond with a type} when we wish to leave the type unspecified.
\end{itemize}
\end{definition}

 %\myframe{Say we do not assign a type to $L_1$, and will not include $L_1$ in counting diamonds of a certain type:
 %for example the second diamond $L_q$ will also be called the first diamond of finite type.}

We refer to diamonds of type $0$ or $1$ as {\em fake diamonds}
to distinguish them from the {\em genuine diamonds} of dimension two.
The necessity of including fake diamonds in a treatment
of Nottingham algebras
arises from the fact that in various notable instances
(as in Theorem~\ref{thm:k=q} below)
diamonds occur at regular intervals, with types following an arithmetic progression.
When such arithmetic progression of types
passes through $0$ or $1$, fake diamonds arise.

However, this carries an inherent ambiguity:
whenever $L_m$ satisfies the definition of a diamond of type $1$
(which amounts to $[L_{m-1}y]=0$ and $[L_mx]=0$), the next homogeneous component $L_{m+1}$ satisfies the definition of a diamond of type $0$
(because $[L_my]=L_{m+1}$ due to the covering property, and then $[L_{m+1}y]=0$ due to $[Lyy]=0$).
Thus, if $w$ spans $L_{m-1}$ then $[wx]$ spans $L_m$
and $[wxy]$ spans $L_{m+1}$, and we have the relations
\begin{equation}\label{eq:fake_diamonds}
[wy]=0, \qquad
[wxx]=0, \qquad
[wxyy]=0.
\end{equation}
The first and second relations are those in part~(a) of
Definition~\ref{def:type-fake}, and the second and third relations
are those in part $(b)$ if we use $w'=[wx]$ instead of $w$ in it.
For various reasons it is inconvenient to simultaneously regard
two consecutive components as fake diamonds, and so we adopt the following convention.

\begin{convention}\label{convention:fake}
% Whenever we are in the presence of fake diamonds,
% with Equation~\eqref{eq:fake_diamonds} holding, where
% $w$ spans $L_{m-1}$,
Whenever we have a diamond $L_m$ of type $1$, necessarily followed by a diamond $L_{m+1}$ of type $0$,
we allow ourselves to call (fake) diamond precisely one of $L_m$
and $L_{m+1}$, of the appropriate type, and not the other.
\end{convention}

% The main reason for Convention~\ref{convention:fake}
% is that we should never
% view $L_m$ and $L_{m+1}$ as fake diamonds {\em simultaneously,}
% as allowing two (fake) diamonds at distance (or degree difference)
% one will disrupt every discussion on diamond distances.

In several cases there is a natural choice between
calling $L_m$ a diamond of type $1$, or $L_{m+1}$ a diamond of type $0$,
which makes diamonds (including the fake ones)
occur at regular distances,
with a difference of $q-1$ in degrees.
%We quote from~\cite{AviMat:diamond_distances} the following
% \begin{theorem}\label{thm:distance_intro}
% Let $L$ be a Nottingham algebra, with second diamond $L_q$,
% and let $y$ be a nonzero element of $L_1$ which centralizes $L_2$.
% Then any two-dimensional component of $L$ can be assigned a type,
% and any one-dimensional component $L_m$ which is not
% centralized by $y$ is a diamond of type $1$.
% Possibly adopting the alternate interpretation of the latter case,
% as $L_{m+1}$ being a diamond of type $0$,
% the difference in degree of any two consecutive diamonds
% may be read to equal $q-1$.
% \end{theorem}

We illustrate that through the following existence result, which will be clarified and expanded in commentaries to follow.

\begin{theorem}\label{thm:k=q}
There exist infinite-dimensional
Nottingham algebras $L$ with second diamond $L_q$,
 %(hence named {\em Nottingham algebras}),
where (possibly fake) diamonds occur in  each degree congruent to $1$ modulo $q-1$, and have types described by any of the following patterns:
\begin{itemize}
\item[(a)]
all diamonds of type $-1$~\cite{Car:Nottingham};
%(existence and uniqueness proved in~\cite{Car:Nottingham});
\item[(b)]
all diamonds of finite types following a non-constant arithmetic progression~\cite{Avi,AviMat:A-Z};
%(existence proved in~\cite{Avi} and~\cite{AviMat:A-Z}, uniqueness in~\cite{CaMa:Nottingham});
%existence proved in~\cite{Avi} for finite types all in the prime field, and~\cite{AviMat:A-Z} otherwise
\item[(c)]
all diamonds of infinite type except for those in degrees $\equiv q\pmod{p^s(q-1)}$ for some $s>0$, which have type $-1$~\cite{AviMat:A-Z};
%(existence proved in~\cite{AviMat:A-Z}, uniqueness proved or not?);
\item[(d)]
all diamonds of infinite type except for those in degrees $\equiv q\pmod{p^s(q-1)}$ for some $s>0$, which have finite types following a non-constant arithmetic progression~\cite{AviMat:mixed_types}.
%(existence proved in~\cite{AviMat:mixed_types}, uniqueness yet unproven);
\end{itemize}
%In all cases, each homogeneous component which is not a diamond or immediately precedes a diamond is centralized by $y$.
%The thin Lie algebras of cases (a) and (b) are uniquely determined by an appropriate finite-dimensional quotient~\cite{Car:Nottingham}.
\end{theorem}

Nottingham algebras as in case~(a) of Theorem~\ref{thm:k=q},
% thus having diamonds in all degrees congruent to $1$ modulo $q-1$,
thus with all diamonds having type $-1$,
were explicitly constructed in~\cite{Car:Nottingham},
using a certain cyclic grading of Zassenhaus algebras.
The special case where $q=p$ is the graded Lie algebra associated with the lower central series of the Nottingham group, thus justifying their name.
They were also shown in~\cite{Car:Nottingham} to be uniquely determined by some finite-dimensional quotient.
Here and in certain other cases such `uniqueness' was proved by exhibiting a finite presentation for some central extension of $L$.
(In most cases $L$ is not itself finitely presented.)
%call it a {\em nearly-finite presentation} in what follows.

Concerning case~(b) of Theorem~\ref{thm:k=q},
Nottingham algebras including fake diamonds were first observed in~\cite{CaMa:Nottingham}).
More precisely, finite presentations for central extensions (and second-central in one case) of Nottingham algebras were given
(with one exception on which we will expand below),
where the diamonds occur in all degrees congruent to $1$ modulo $q-1$,
and their types follow a non-constant arithmetic progression.
If that arithmetic progression passes through $0$, that is, if it runs through the prime field,
then those diamonds include fake diamonds, of both types $0$ and $1$.
Such Nottingham algebras were explicitly constructed, thus proving their existence,
in~\cite{Avi} in case all types belong to the prime field, and in~\cite{AviMat:A-Z} otherwise.
Those constructions use certain finite-dimensional simple modular Lie algebras of Cartan type,
and certain gradings of them over a finite cyclic group.

Nottingham algebras where the third diamond has infinite type include those of cases~(c) and~(d).
Again, their constructions in~\cite{AviMat:A-Z} and~\cite{AviMat:mixed_types} used
certain finite-dimensional simple modular Lie algebras of Cartan type, but special tools involving generalized exponentials of derivations
had to be developed for producing the required gradings,
in~\cite{Mat:Artin-Hasse,AviMat:Laguerre,AviMat:gradings}.
Proving uniqueness of those Nottingham algebras is one of the goals of the present paper, in Theorem~\ref{thm:fin_pres},
which implies Theorem~\ref{thm:presentation}.

In all cases of Theorem~\ref{thm:k=q}, each homogeneous component which is not a diamond or immediately precedes a diamond is centralized by $y$.
Together with this information,
the locations and types of all diamonds,
as specified in each case of Theorem~\ref{thm:k=q},
give a complete description of those Nottingham algebras.
% With that assumption, specifying all degrees
% in which diamonds occur (possibly fake and making use of Convention~\ref{convention:fake}),
% and their types, describes a Nottingham algebra completely.
Note that each of the Nottingham algebras of Theorem~\ref{thm:k=q} has diamonds in each degree congruent to $1$ modulo $q-1$.
In particular, the distance between consecutive diamonds in those particular Nottingham algebras
is invariably $q-1$, provided that we assign
an appropriate type $0$ or $1$ to each fake diamond
(making use of Convention~\ref{convention:fake}).
Several such features remain true for arbitrary Nottingham algebras,
as we discuss in the next section.

All Nottingham algebras of Theorem~\ref{thm:k=q}
display a periodic structure which, however,
is not a universal characteristic
of Nottingham algebras.
In fact, in his PhD thesis~\cite{Young:thesis}
David Young gave two procedures which allow
one to produce two Nottingham algebras
$\mathcal{T}_{q,1}(M)$ and $\mathcal{T}_{q,2}(M)$, both with second diamond $L_q$,
starting from any given graded Lie algebra $M$ of maximal class with at most two distinct two-step centralizers (see~\cite{CMN,CN} for the latter).
The diamond patterns of
$\mathcal{T}_{q,1}(M)$
and
$\mathcal{T}_{q,2}(M)$
reflect the pattern of
two-step centralizers of $M$, in two different ways.
Because, over any given field of characteristic $p$,
there are uncountably many
such algebras $M$, and most of them are not periodic,
corresponding assertions carry over to the class of Nottingham algebras, over a given field of characteristic $p$ and for a fixed power $q$ of $p$.
Both algebras
$\mathcal{T}_{q,1}(M)$
and
$\mathcal{T}_{q,2}(M)$
have second diamond $L_q$ (of type $-1$ by definition),
and all remaining diamonds are fake or have infinite type.
% In case of $L=\mathcal{T}_{q,1}(M)$
% the second diamond $L_q$ is followed by a fake diamond
% $L_{2q-1}$, hence such algebras are not relevant
% to the main goals of this paper, and we will focus here on
% $\mathcal{T}_{q,2}(M)$.
We focus here on $\mathcal{T}_{q,2}(M)$,
which is the one relevant to this paper.

Diamonds of infinite type in $L=\mathcal{T}_{q,2}(M)$
occur in sequences, of lengths dictated
in a certain way by the structure of $M$,
separated by single occurrences of fake diamonds.
The diamonds of infinite type in those sequences
occur at distances of $q-1$ in degree.
% (In case of the very first sequence, which starts with $L_{2q-1}$
% this assertion may as well include $L_1$ and $L_q$.)
However, if $L_m$ is a diamond ending any such sequence,
then $L_{m+q-1}$ is a (fake) diamond of type $1$,
and then $L_{m+2q-1}$ begins the next sequence
of diamonds of infinite type.
Thus, the degree difference between $L_{m+q-1}$ and $L_{m+2q-1}$
equals $q$, rather than $q-1$ as in the examples
from Theorem~\ref{thm:k=q} which we discussed earlier.
However, if we make use of the ambiguity
which is intrinsic in the definition of fake diamonds,
and view $L_{m+q}$ as a diamond of type $0$, then that has distance
$q-1$ from the next diamond $L_{m+2q-1}$.
In conclusion, the existence of $\mathcal{T}_{q,2}(M)$
does not contradict a general claim
(which is a theorem in~\cite{AviMat:diamond_distances},
see Section~\ref{sec:main} below)
that the distance between two consecutive diamonds
of a Nottingham algebra may always be interpreted
to equal $q-1$,
provided that in the presence of fake diamonds
we allow ourselves to choose which component we call fake diamond,
according to Convention~\ref{convention:fake}.
% Such a claim is indeed proved in~\cite{AviMat:diamond_distances},
% as we will explain in the next section.
We stress that, differently from the algebras considered in
Theorem~\ref{thm:k=q}, a fake diamond in $\mathcal{T}_{q,2}(M)$
needs to be interpreted
in two different ways (with the corresponding shift by one in degree),
according to which distance we intend to measure
(whether from the previous or to the next diamond).
This double interpretation of the same fake diamond
is required in $\mathcal{T}_{q,1}(M)$ as well,
and also in other Nottingham algebras studied in~\cite{Young:thesis}.

\section{The degree of the second diamond of finite type }\label{sec:main}

It is not at all obvious that a two-dimensional component
$L_m$ of an arbitrary Nottingham algebra, with $m>q$,
should satisfy the relations of Definition~\ref{def:type} for some $\mu$, thus allowing
type $\mu$ to be assigned to it.
In fact, this is one of the main conclusions of~\cite{AviMat:diamond_distances}.
More generally, it is shown there that whenever
$[L_{m-2}y]=0$ and $[L_{m-1}y]\neq 0$ for some $m>q$,
either $L_m$ is two-dimensional and can be assigned a
type $\mu$ according to Definition~\ref{def:type},
or $L_m$ is a (fake) diamond of type $0$
according to Definition~\ref{def:type-fake}.
As we have observed right after that definition,
the latter situation
admits the alternate interpretation that
$L_{m-1}$ is a diamond of type $1$.

Another main result of~\cite{AviMat:diamond_distances}
is that any two consecutive diamonds in an arbitrary Nottingham algebra can always be assumed to have a difference of $q-1$
in degrees, allowing appropriate interpretation in case
fake diamonds are involved.
We refer the reader
to~\cite{AviMat:diamond_distances}
for a deeper discussion of this and
further results,
and state here only the portions
which we need in this paper.

% We formally state that here in a formulation
% which is convenient for our purposes.

% *This takes care of the distance between two consecutive diamonds. Because of the ambiguity this is not sufficient to determine the relative distances between three consecutive diamonds, and we need an additional result from... for that.
% *

\begin{theorem}\label{thm:distance}
Let $L$ be a Nottingham algebra with second diamond $L_q$, and standard generators $x$ and $y$.
Let $L_m$ be a (possibly fake) diamond of $L$, with $m\ge q$.
\begin{itemize}
\item[(a)]
If $L_m$ is a genuine diamond then $L_{m+q-1}$ is a diamond.
\item[(b)]
If $L_m$ is a diamond of type $1$, then either $L_{m+q-1}$
or $L_{m+q}$ is a diamond.
\item[(c)]
If $L_m$ is a diamond of type $0$ and, in addition,
$L_{m-q+1}$ is a diamond of type different from $0$,
then $L_{m+q-1}$ is a diamond.
\end{itemize}
Furthermore, in each case $y$ centralizes $L_{m+1},\ldots,L_{m+q-3}$,
and also $L_{m+q-2}$ if $L_{m+q-1}$ is not a diamond in assertion~(b).
\end{theorem}

Various clarifications are in order.
The final assertion of
Theorem~\ref{thm:distance},
together with assertion~(a) and~(b),
show that a Nottingham algebra
$y$ centralizes each homogeneous
component which is not
a diamond or immediately
precedes a diamond.
Thus, the locations and types of the diamonds
(still making use of Convention~\ref{convention:fake})
suffice to completely describe an
arbitrary Nottingham algebra.
Next, the two conclusions of assertion~(b)
are not disjoint, the common case being
when $L_{m+q-1}$ is a diamond of type $1$,
which means the same as $L_{m+q}$
being a diamond of type $0$.
Also, the hypothesis of assertion~(b)
can be alternately read as $L_{m+1}$
being a diamond of type $0$.
Altogether, we see that the difference
in degree between consecutive diamonds
can always be read as $q-1$,
as long as we suitably interpret
any fake diamonds involved.
Finally, note that if we read the hypothesis of assertion~(c)
as $L_{m-1}$ being a diamond of type $1$, then assertion~(b)
would only tell us that either $L_{m+q-2}$ or $L_{m+q-1}$
is a diamond.
In fact, inferring the stronger
conclusion of assertion~(c)
requires information on the diamond
which precedes $L_m$.
All assertions of Theorem~\ref{thm:distance} were stated
in~\cite{AviMat:diamond_distances} under a blanket hypothesis
that Nottingham algebras have infinite dimension.
However, as pointed out there, those assertions remain true
for a finite-dimensional Nottingham algebra $L$
as long as none of the homogeneous components
they mention is the last nonzero
homogeneous component of $L$ or the preceding one.

After recalling such general features of Nottingham algebras
established in~\cite{AviMat:diamond_distances}, we focus on
a major goal of this paper, Theorem~\ref{thm:main_intro},
which determines the possible degrees $t$ in which the
(possibly fake) next diamond of finite type $L_t$ may occur past $L_q$ and a sequence of diamonds of infinite type in a Nottingham algebra $L$.
We state here a more precise version of
Theorem~\ref{thm:main_intro}.
While that more concise result assumes $L$ to have infinite dimension, we relax that assumption in Theorem~\ref{thm:main},
thus turning the goal into proving finite-dimensionality of $L$ unless the degree of that diamond of finite type has the particular form claimed in Theorem~\ref{thm:main_intro}.

According to Theorem~\ref{thm:distance},
all distances between consecutive diamonds up to $L_t$ equal $q-1$, and hence $t\equiv 1\pmod{q-1}$.

\begin{theorem}\label{thm:main}
Let $L$ be a Nottingham algebra with second diamond $L_q$ and standard generators $x$ and $y$.
%(necessarily of type $-1$)
%in degree $q=p^n\geq 5$, for some $n\geq 1$.
Suppose $L_{2q-1}$ is a diamond of infinite type.
% and suppose that $y$ centralizes every
% homogeneous component from $L_{q+1}$ up to $L_{2q-3}$.
Suppose that the earliest diamond of finite type past $L_{q}$
occurs in degree $a(q-1)+1$, and has type $\mu\in \F$.
%Let $a\geq 3$ be an integer. Suppose that $L$
%has diamonds of infinite type in each degree $k(q-1)+1$, $2\leq k \leq a-1$,
%and a (possibly fake) diamond of finite type $\mu \in \F$ in degree $a(q-1)+1$.
Then the following assertions hold:
\begin{itemize}
\item[(a)]
$a$ is even;
\item[(b)]
if $a \not \equiv 1 \pmod{p}$ then $L_{(a+1)(q-1)+3}=0$, unless
$a\equiv 0\pmod{p}$, $\mu=1$, and $[L_{(a+1)(q-1)},y]=0$;
\item[(c)]
if $a\equiv 1 \pmod{p}$ but $a-1$ is not a power of $p$, then
$L_{(a+p^s)(q-1)+2}=0$, where $p^s$ is the highest power of $p$ which divides $a-1$;
\item[(d)]
if $a\equiv 0\pmod{p}$, $\mu=1$, $[L_{(a+1)(q-1)},y]=0$, and $a/2$ is not a power of $p$,
then $L_{(a+p^s)(q-1)+3}=0$, where $p^s$ is the highest power of $p$ which divides $a$.
\end{itemize}
\end{theorem}

If the Nottingham algebra $L$ of
Theorem~\ref{thm:main}
has infinite dimension, it follows that
the next diamond of finite type past $L_q$
has degree $a(q-1)+1$ where either $a-1$
is a power of $p$ greater than $1$
(coming from Assertion~(c)),
or $a$ equals twice a power of $p$
greater than $1$ (coming from Assertion~(d)).
This is the content of Theorem~\ref{thm:main_intro}.
The former case occurs, in particular,
for the Nottingham algebras of
cases~(c) and~(d) of Theorem~\ref{thm:k=q}.
The latter case occurs for the family of algebras
$\mathcal{T}_{q,2}(M)$
found by David Young in~\cite{Young:thesis},
which we briefly introduced
at the end of the previous section.
In fact, if the graded Lie algebra $M$ of maximal class,
with at most two distinct two-step centralizers,
from which $L=\mathcal{T}_{q,2}(M)$ is constructed,
has {\em first constituent} of length $2p^s$
(as defined in~\cite{CN}),
then the earliest diamond of finite type past $L_q$
occurs as fake of type $1$ in degree $2p^s(q-1)+1$.

The proof of Theorem~\ref{thm:main} is rather long and occupies the whole of Section~\ref{sec:proof_main}.
Before embarking in the proof proper, in the next section
we set up some notation and perform some preliminary calculations
in an arbitrary Nottingham algebra,
which will also be useful in proving our second main result, in Section~\ref{sec:fin_pres}.

\section{General calculations near the diamonds}\label{sec:calculations}

Throughout the paper we make extensive use of the {\em generalized Jacobi identity}
\[
[a[bc^n]]=\sum_{i=0}^{n} (-1)^i \binom{n}{i} [ac^{i}bc^{n-i}].
\]
Two special instances which often occur are $[a[bc^q]]=[abc^q]-[ac^qb]$ (which amounts to $(\ad c)^q$ being a derivation), and
$
[a[bc^{q-1}]]=\sum_{i=0}^{q-1}\;[ac^{i}bc^{q-1-i}],
$
due to $\binom{q-1}{i}\equiv (-1)^i\pmod{p}$.
More generally, the binomial coefficients involved in the generalized Jacobi identity can be efficiently evaluated modulo $p$ by means of Lucas' theorem:
%, see \cite[p. 271]{Dickson1}:
if $q$ is a power of $p$ and  $a,b,c,d$ are non-negative
integers with $b,d<q$, then
$
\binom{aq+b}{cq+d}\equiv \binom{a}{c}\binom{b}{d} \pmod p.
$

Now consider a Nottingham algebra $L$ with second diamond in degree $q$.
Besides $v_{1}=[yx^{q-2}]$ as in the previous section, we set $v_{2}=[v_{1}xyx^{q-3}]$.
Note that, according to Theorem~\ref{thm:distance},
$y$ centralizes $L_{q+1},\ldots,L_{2q-3}$,
and $L_{2q-1}$ is a diamond, possibly fake,
spanned by $[v_2x]$ and $[v_2y]$.
In the rest of the paper we will use $v_k$ to denote a certain element spanning the homogeneous component which immediately precedes the $(k+1)$st diamond.
Although in some previous papers a notation with the subscript increased by one was used,
our choice appears slightly more convenient because $v_k$ will have degree $k(q-1)$.

Suppose $L_m$ is a diamond of arbitrary type $\mu\in\F\cup\{\infty\}$,
and assume $v_k$ spans $L_{m-1}$.
Assume $L_{m+q-1}$ is a diamond
(which is a consequence of
Theorem~\ref{thm:distance}
only when $L_m$ is genuine).
Then according to
Theorem~\ref{thm:distance}
the element $y$ centralizes $L_{m+1},\ldots,L_{m+q-3}$
Define the element $v_{k+1}$ in degree
$m+q-2$ as
\[
v_{k+1}=\begin{cases}
[v_{k}xyx^{q-3}] &\mbox{if } \mu \neq 0, \\
[v_{k}yx^{q-2}] &\mbox{otherwise}.
\end{cases}
\]
We start with describing the adjoint action of $v_{1}$ on the elements close to this diamond.
We use the convention $\infty^{-1}=0$.

\begin{lemma}\label{lemma:v1}
Suppose $[v_{k}yx]=(\mu^{-1}-1)[v_{k}xy]$ with $\mu \in \F^{\ast}\cup\{\infty\}$.
Then we have
\begin{align*}
&[v_{k}v_{1}]=(\mu^{-1}+1)v_{k+1}, \\
&[v_{k}xv_{1}]=[v_{k+1}x], \\
&[v_{k}yv_{1}]=(1-\mu^{-1})[v_{k+1}y], \\
&[v_{k}xyv_{1}]=-(2[v_{k+1}yx]+[v_{k+1}xy]), \\
&[v_{k}xyxv_{1}]=-(3[v_{k+1}yx^2]+2[v_{k+1}xyx]).
\end{align*}
\end{lemma}
\begin{proof}
All the claimed equations are easy to prove using the generalized Jacobi identity, so we prove only a couple of them.
The first equation is
\[
[v_{k}v_{1}]=[v_{k}[yx^{q-2}]]=[v_{k}yx^{q-2}]+2[v_{k}xyx^{q-3}]=(\mu^{-1}+1)v_{k+1}.
\]
The third equation is
\[
[v_{k}yv_{1}]=[v_{k}y[yx^{q-2}]]
=-[v_{k}yx^{q-2}y]=(1-\mu^{-1})[v_{k+1}y].
\]
\end{proof}
\begin{rem}\label{rem:v1_mu_0}
In the case $\mu=0$ excluded from Lemma~\ref{lemma:v1}, similar calculations show
$[v_{k}v_{1}]=v_{k+1}$, $[v_{k}yv_{1}]=-[v_{k+1}y]$, $[v_{k}yxv_{1}]=-(2[v_{k+1}yx]+[v_{k+1}xy])$ and
$[v_{k}yx^2v_{1}]=-(3[v_{k+1}yx^2]+2[v_{k+1}xyx])$.
\end{rem}

As an immediate consequence of Lemma~\ref{lemma:v1} we obtain the adjoint action of
$[v_{1}x]$ on elements close to diamonds.
\begin{cor}\label{cor:v1x}
Under the hypotheses of Lemma~\ref{lemma:v1} we have
\begin{align*}
& [v_{k}[v_{1}x]]=\mu^{-1}[v_{k+1}x], \\
&[v_{k}x[v_{1}x]]=0,  \\
&[v_{k}y[v_{1}x]]=(\mu^{-1}-1)([v_{k+1}yx]+[v_{k+1}xy]),\\
&[v_{k}xy[v_{1}x]]=[v_{k+1}yx^2]+[v_{k+1}xyx].
\end{align*}
\end{cor}

\begin{rem}\label{rem:v1x_mu_0}
In the excluded case $\mu=0$ we find
$[v_{k}[v_{1}x]]=[v_{k+1}x]$, $[v_{k}y[v_{1}x]]=[v_{k+1}yx]+[v_{k+1}xy]$ and
$[v_{k}yx[v_{1}x]]=[v_{k+1}yx^2]+[v_{k+1}xyx]$.
\end{rem}

Now assume the diamond $L_{m+q-1}$ has infinite type, and set $v_{k+2}=[v_{k+1}xyx^{q-3}]$.
According to Theorem~\ref{thm:distance},
the element $y$ centralizes $L_{m+q},\ldots,L_{m+2q-4}$,
and $L_{m+2q-2}$ is a diamond.
We describe the adjoint action of $v_{2}$ on elements close to the diamond $L_m$.

\begin{lemma}\label{lemma:v2}
Suppose $[v_{k}yx]=(\mu^{-1}-1)[v_{k}xy]$, with $\mu\in \F^{\ast}\cup\{\infty\}$, and $[v_{k+1}yx]=-[v_{k+1}xy]$.
Then we have
\begin{align*}
&[v_{k}v_{2}]=\mu^{-1}v_{k+2}, \quad [v_{k}xv_{2}]=0=[v_{k}yv_{2}], \\
&[v_{k}xyv_{2}]=[v_{k+2}yx]+[v_{k+2}xy],\\
&[v_{k}xyxv_{2}]=2([v_{k+2}yx^2]+[v_{k+2}xyx]).
\end{align*}
Furthermore,  we have
\begin{align*}
&[v_{k}xyx^{q-4}v_{1}]=[v_{k+1}xyx^{q-4}], \quad [v_{k}xyx^{q-4}[v_{1}x]]=0, \\
&[v_{k}xyx^{q-4}v_{2}]=-3([v_{k+2}yx^{q-3}]+[v_{k+2}xyx^{q-4}]).
\end{align*}
\end{lemma}

\begin{proof}
Using Corollary~\ref{cor:v1x}, the first equation in the first set is
\begin{align*}
[v_{k}v_{2}]&=[v_{k}[v_{1}xyx^{q-3}]]=[v_{k}[v_{1}xy]x^{q-3}]+3[v_{k}x[v_{1}xy]x^{q-4}]\\
&=[v_{k}[v_{1}x]yx^{q-3}]-[v_{k}y[v_{1}x]x^{q-3}]-3[v_{k}xy[v_{1}x]x^{q-4}]\\
&=\mu^{-1}v_{k+2}.
\end{align*}
The first equation in the second set is
\begin{align*}
&[v_{k}xyx^{q-4}v_{1}]=[v_{k}xyx^{q-4}[yx^{q-2}]]\\
&=2[v_{k+1}yx^{q-3}]+3[v_{k+1}xyx^{q-4}]=[v_{k+1}xyx^{q-4}].
\end{align*}
Similar calculations yield the remaining equations.
\end{proof}

\begin{rem}\label{rem:v2_mu_0}
In the case $\mu=0$, which is excluded from
Lemma~\ref{lemma:v2}, we find
$[v_{k}v_{2}]=v_{k+2}$, $[v_{k}yv_{2}]=0$, $[v_{k}yxv_{2}]=[v_{k+2}yx]+[v_{k+2}xy]$
and $[v_{k}yx^{2}v_{2}]=2([v_{k+2}yx^2]+[v_{k+2}xyx])$.
\end{rem}

Now assume the diamond $L_m$ has infinite type, and that the diamond $L_{m+q-1}$ has finite type $\mu$.
Set $v_{k+2}=[v_{k+1}xyx^{q-3}]$, unless $\mu=0$, in which case set $v_{k+2}=[v_{k+1}yx^{q-2}]$.
According to Theorem~\ref{thm:distance},
the element
$y$ centralizes $L_{m+q},\ldots,L_{m+2q-4}$.
Assume $L_{m+2q-2}$ is a diamond
(which is a consequence of Theorem~\ref{thm:distance}
only if $\mu\neq 1$).
In case $\mu=1$ we do assume here that $L_{m+2q-2}$ is a diamond.
We describe the adjoint action of $v_{2}$ on the elements close to the diamond $L_m$.

\begin{lemma}\label{lemma:v2ext}
Suppose $[v_{k}xy]=-[v_{k}yx]$  and  $[v_{k+1}yx]=(\mu^{-1}-1)[v_{k+1}xy]$ for some $\mu\in \F^{\ast}$.
Then we have
\begin{align*}
&[v_{k}v_{2}]=-2\mu^{-1} v_{k+2},\\
&[v_{k}xv_{2}]=-\mu^{-1} [v_{k+2}x],\\
&[v_{k}yv_{2}]=-\mu^{-1} [v_{k+2}y], \\
&[v_{k}xyv_{2}]=[v_{k+2}xy]+(2\mu^{-1}+1)[v_{k+2}yx],\\
&[v_{k}xyxv_{2}]=2[v_{k+2}xyx]+(3\mu^{-1}+2)[v_{k+2}yx^2].
\end{align*}
\end{lemma}

\begin{proof}
The first equation is
\begin{align*}
&[v_{k}v_{2}]=[v_{k}[v_{1}xy]x^{q-3}]+3[v_{k}x[v_{1}xy]x^{q-4}]\\
&=-[v_{k}y[v_{1}x]x^{q-3}]-3[v_{k}xy[v_{1}x]x^{q-4}]\\
&=-2([v_{k+1}yx^{q-2}]+v_{k+2})=-2\mu^{-1}v_{k+2}.
\end{align*}
The remaining equations can be proved similarly.
\end{proof}

\begin{rem}\label{rem:v2ext_mu_0}
In the excluded case $\mu=0$ we find
$[v_{k}v_{2}]=-2v_{k+2}$, $[v_{k}xv_{2}]=-[v_{k+2}x]$,
$[v_{k}yv_{2}]=-[v_{k+2}y]$, $[v_{k}xyv_{2}]=2[v_{k+2}yx]$ and $[v_{k}xyxv_{2}]=3[v_{k+2}yx^2]$.
\end{rem}

\section{Proof of Theorem~\ref{thm:main}}\label{sec:proof_main}

%Recall $v_{1}=[yx^{q-2}]$ and,
Define recursively $v_{k+1}=[v_{k}xyx^{q-3}]$ for $1\leq k \leq a-1$.
 %In particular, $v_{2}=[v_{1}xyx^{q-3}]$.
Further, define the element $v_{a+1}$ in degree
$(a+1)(q-1)$ as
\[
v_{a+1}=\begin{cases}
[v_{a}xyx^{q-3}] &\mbox{if } \mu \neq 0, \\
[v_{a}yx^{q-2}] &\mbox{otherwise}.
\end{cases}
\]
(However, $v_{a+1}$ will be redefined in the proof of assertion $4$.)
Note that $v_{k}$ has degree $k(q-1)$.
Theorem~\ref{thm:distance} can be inductively used
to show that $[L_{i},y]=0$ except when $i$ is congruent to $0$ or $1$ modulo $q-1$.

\subsection{Proving that $a$ is even}\label{subsec:a_even}

According to the first equation proved in Lemma~\ref{lemma:v1}
the recursive definition of the elements $v_{k}$, for $3\leq k \leq a$, can be replaced with the more compact formula
$v_{k}=[v_{2}v_{1}^{k-2}]$. In particular, this allow us, together with the equations in Lemmas~\ref{lemma:v1}, \ref{lemma:v2}, and~\ref{lemma:v2ext}, to compute
the adjoint action of any $v_{k}$ on homogeneous elements close to diamonds. A first instance of this type of
calculation occurs in the proof of the following result.

\begin{prop}\label{prop:a_odd}
Under the hypotheses of Theorem~\ref{thm:main}, if $a$ is odd then the relation
$
[v_{a}xy]+[v_{a}yx]=0
$
holds.
\end{prop}

\begin{proof}
We expand both sides of the equation
\[
[v_{a-1}[v_{1}xy]]=-[v_{1}xyv_{a-1}]
\]
using Corollary~\ref{cor:v1x} and Lemmas~\ref{lemma:v1} and~\ref{lemma:v2}.
The left-hand side is
\[
[v_{a-1}[v_{1}xy]]=[v_{a-1}[v_{1}x]y]-[v_{a-1}y[v_{1}x]]=-[v_{a-1}y[v_{1}x]]=[v_{a}yx]+[v_{a}xy].
\]
Expanding the right-hand side we have
\begin{align*}
&[v_{1}xyv_{a-1}]=[v_{1}xy[v_{2}v_{1}^{a-3}]=\sum_{i=0}^{a-3}(-1)^i\binom{a-3}{i}[v_{1}xyv_{1}^{i}v_{2}v_{1}^{a-3-i}]\\
&=\sum_{i=0}^{a-3}(-1)^i\binom{a-3}{i}[v_{1+i}xyv_{2}v_{1}^{a-3-i}]=(-1)^{a-3}[v_{a-2}xyv_{2}]\\
&=(-1)^{a-3}([v_{a}yx]+[v_{a}xy]).
\end{align*}
Having assumed $a$ odd we conclude that $[v_{a}yx]+[v_{a}xy]=0$.
\end{proof}

The relation $[v_{a}xy]+[v_{a}yx]=0$ proved in Proposition~\ref{prop:a_odd} contradicts the assumption that $L_{a(q-1)+1}$ has finite type. Hence
$a$ must be even, and Assertion~(a) of
Theorem~\ref{thm:main} is proved.

\subsection{The case where $a$ is not congruent to $1$ modulo $p$.}\label{subsec:a_not_equiv_1}

By hypothesis our algebra $L$ has a diamond of finite type $\mu \in \F$ in degree $a(q-1)+1$, for some even integer $a>2$,
and diamonds of infinite type in all
lower degrees congruent to $1$ modulo $q-1$, with the exception of $q$.
We also know, from Theorem~\ref{thm:distance}, that
$y$ centralizes all homogeneous component from $L_{a(q-1)+2}$ up to $L_{a(q-1)+q-2}$.
In this subsection we gather some information on $L$
up to the next diamond after $L_{a(q-1)+1}$
and a little past that.
In particular, this will provide a proof
of Assertion~(b) of Theorem~\ref{thm:main}.

In this subsection we assume
$L_{(a+1)(q-1)+3}\neq 0$.
Suppose first that $L_{a(q-1)+1}$ is a diamond of type $\mu \neq 1$.
Again according to Theorem~\ref{thm:distance}, the homogeneous component $L_{(a+1)(q-1)+1}$ is a diamond.
%We need not determine the type of $L_{(a+1)(q-1)+1}$ at this stage.
% Consequently,
% $[v_{a+1}xyx]$ and $[v_{a+1}yx^2]$
% span $L_{(a+1)(q-1)+3}$.
We now prove that it has infinite type,
which amounts to the vanishing of
$[v_{a+1}yx]+[v_{a+1}xy]$.
Because of the covering property, it suffices
to show that this element belongs to the centre.
Being necessarily centralized by $y$ because of
$(\ad y)^2=0$,
we only need to show $[v_{a+1}yx^2]=-[v_{a+1}xyx]$.

Since $[v_{2}yx]+[v_{2}xy]=0$ we have
\[
[v_{a-1}x[v_{2}yx]]+[v_{a-1}x[v_{2}xy]]=0.
\]
Assume first $\mu\neq 0$.
We expand both terms on the left-hand side  by means of Lemma~\ref{lemma:v2ext}.
The first term is
\begin{equation}\label{eq:v_a-1x_A}
\begin{aligned}[c]
[v_{a-1}x[v_{2}yx]]&=[v_{a-1}x[v_{2}y]x]=[v_{a-1}xv_{2}yx]-[v_{a-1}xyv_{2}x]\\
&=-(\mu^{-1}+1)[v_{a+1}xyx]-(2\mu^{-1}+1)[v_{a+1}yx^2],
\end{aligned}
\end{equation}
and the second term is
\begin{align*}
[v_{a-1}x[v_{2}xy]]&=[v_{a-1}x[v_{2}x]y]-[v_{a-1}xy[v_{2}x]]\\
&=[v_{a-1}xv_{2}xy]-[v_{a-1}xyv_{2}x]+[v_{a-1}xyxv_{2}]\\
&=[v_{a+1}xyx]+(\mu^{-1}+1)[v_{a+1}yx^2].
\end{align*}
Putting both terms together shows
$[v_{a+1}yx^2]=-[v_{a+1}xyx]$.
If $\mu=0$ we still consider the equation $[v_{a-1}x[v_{2}yx]]+[v_{a-1}x[v_{2}xy]]=0$.
According to Remark~\ref{rem:v2ext_mu_0}, the first term
on the left-hand side is
\begin{equation}\label{eq:v_a-1x_B}
\begin{aligned}[c]
[v_{a-1}x[v_{2}yx]]=[v_{a-1}xv_{2}yx]-[v_{a-1}xyv_{2}x]=-[v_{a+1}xyx]-2[v_{a+1}yx^{2}].
\end{aligned}
\end{equation}
%since $[v_{a-1}xv_{2}]=-[v_{a+1}x]$ and $[v_{a-1}xyv_{2}]=2[v_{a+1}yx]$.
The second term is
\[
[v_{a-1}x[v_{2}xy]]=[v_{a-1}xv_{2}xy]-[v_{a-1}xyv_{2}x]+[v_{a-1}xyxv_{2}]=[v_{a+1}yx^{2}].
\]
%since $[v_{a-1}xyxv_{2}]=3[v_{a+1}yx^2]$.
We deduce that $[v_{a+1}yx^2]=-[v_{a+1}xyx]$ also in this case.
Thus, we have shown that if $\mu\neq 1$ then the diamond
$L_{(a+1)(q-1)+1}$ has infinite type.

%\begin{rem}\label{rem:mu_1_v_a+1}

Now we consider the case $\mu=1$.
The equation
$[v_{a-1}x[v_{2}yx]]=-[v_{a-1}x[v_{2}xy]]$ yields
\[
[v_{a+1}x^2y]=-[v_{a+1}xyx]-[v_{a+1}yx^{2}].
\]
According to Theorem~\ref{thm:distance}
either $L_{(a+1)(q-1)+1}$ or $L_{(a+1)(q-1)+2}$ is a diamond.
If $[v_{a+1}y]\neq 0$, then we are in the former case,
and so we have  $[v_{a+1}xx]=0$.
% Theorem~\ref{cor:next_diamond}
% tells us that $[v_{a+1}x]$ and $[v_{a+1}y]$
% span a diamond, whence $[v_{a+1}xx]=0$.
Then the above equation tells us that
$[v_{a+1}xy]+[v_{a+1}yx]$ is centralized by $x$.
Hence the diamond
$L_{(a+1)(q-1)+1}$ has, again, infinite type.

Assume now $[v_{a+1}y]=0$, so that
$[v_{a+1}x^2y]=-[v_{a+1}xyx]$.
Note that this element spans $L_{(a+1)(q-1)+3}$
according to Theorem~\ref{thm:distance}.
We will show $a \equiv 0 \pmod{p}$.
Consider the equation $[v_{2}xy[v_{a-1}x]]=[v_{a-1}x[v_{2}yx]]$. We expand both
sides by means of Lemmas~\ref{lemma:v2} and~\ref{lemma:v2ext}. The right-hand side yields
\[
[v_{a-1}x[v_{2}yx]]=[v_{a-1}xv_{2}yx]-[v_{a-1}xyv_{2}x]=-2[v_{a+1}xyx].
\]
The left-hand side yields
$
[v_{2}xy[v_{a-1}x]]=[v_{2}xyv_{a-1}x]-[v_{2}xyxv_{a-1}].
$
The first term of the difference is
\begin{align*}
&[v_{2}xyv_{a-1}x]=[v_{2}xy[v_{2}v_{1}^{a-3}]x]=(a-3)[v_{a-2}xyv_{2}v_{1}x]-[v_{a-1}xyv_{2}x]\\
&=-(a-2)[v_{a+1}xyx].
\end{align*}
The second term of the difference is
\begin{align*}
&[v_{2}xyxv_{a-1}]=[v_{2}xyx[v_{2}v_{1}^{a-3}]]=(a-3)[v_{a-2}xyxv_{2}v_{1}]-[v_{a-1}xyxv_{2}]\\
&=2(a-3)[v_{a}xyxv_{1}]-2[v_{a+1}xyx]=-2(a-3)[v_{a+1}xyx]-2[v_{a+1}xyx],
\end{align*}
where we have used the identity $[v_{a}xyx v_{1}]=-2[v_{a+1}xyx]-[v_{a+1}x^2y]=-[v_{a+1}xyx]$.
Therefore, we have
\[
[v_{2}xy[v_{a-1}x]]=(a-2)[v_{a+1}xyx]
\]
and so
\[
a[v_{a+1}xyx]=0,
\]
whence $a \equiv 0 \pmod{p}$ as claimed.
Summarizing, in case $\mu=1$ we have found
that either $L_{(a+1)(q-1)+1}$ is a diamond of infinite type
(same as in the case $\mu\neq 1$),
or $L_{(a+1)(q-1)+2}$ is a diamond, in which
case $a \equiv 0 \pmod{p}$.

We now proceed to prove Assertion (b) of
Theorem~\ref{thm:main}.

\begin{prop}\label{prop:a_not_equiv_1}
Under the hypotheses of Theorem~\ref{thm:main}, if $a\not\equiv 1 \pmod{p}$ then
$L_{(a+1)(q-1)+3}=0$, unless
$a \equiv 0 \pmod{p}$, $\mu=1$, and $[L_{(a+1)(q-1)},y]=0$.
\end{prop}

\begin{proof}
We keep with the previous assumption that $L_{(a+1)(q-1)+3}\neq 0$.
We have already proved that
if $[L_{(a+1)(q-1)},y]=0$ then
$a\equiv 0\pmod{p}$ and $\mu=1$,
and so we may assume
% $[L_{(a+1)(q-1)},y]\neq 0$ whence
$L_{(a+1)(q-1)+1}$ to be a diamond of infinite type.
Hence $[v_{a+1}yx]=-[v_{a+1}xy]$,
and
$[v_{a+1}xyx]\neq 0$.

We consider the equation
\[
[v_{2}xy[v_{a-1}x]]=-[v_{a-1}x[v_{2}xy]] =[v_{a-1}x[v_{2}yx]]
\]
and expand the first and the last Lie products by means of Lemmas~\ref{lemma:v2} and~\ref{lemma:v2ext}.
We need to distinguish two cases according as to whether
$\mu$ equals zero or not.
According to Equations~\eqref{eq:v_a-1x_A}
and~\eqref{eq:v_a-1x_B} we have
$[v_{a-1}x[v_{2}yx]]=\mu^{-1}[v_{a+1}xyx]$ if $\mu\neq 0$,
and
$[v_{a-1}x[v_{2}yx]]=[v_{a+1}xyx]$ if $\mu=0$.
As to the other term  we have
\[
[v_{2}xy[v_{a-1}x]]=[v_{2}xyv_{a-1}x]-[v_{2}xyxv_{a-1}].
\]

If $\mu \neq 0$, the former term in this difference is
\begin{align*}
&[v_{2}xyv_{a-1}x]=[v_{2}xy[v_{2}v_{1}^{a-3}]x]
=(a-3)[v_{a-2}xyv_{2}v_{1}x]-[v_{a-1}xyv_{2}x]\\
&=(a-3)([v_{a}yxv_{1}x]+[v_{a}xyv_{1}x])-(2 \mu^{-1}+1)[v_{a+1}yx^2]-[v_{a+1}xyx]\\
&=(a-3)\mu^{-1}[v_{a+1}xyx]+2\mu^{-1}[v_{a+1}xyx]\\
&=(a-1)\mu^{-1}[v_{a+1}xyx].
\end{align*}
The latter term is
\begin{align*}
&[v_{2}xyxv_{a-1}]=[v_{2}xyx[v_{2}v_{1}^{a-3}]]
=(a-3)[v_{a-2}xyxv_{2}v_{1}]-[v_{a-1}xyxv_{2}]\\
&=2(a-3)([v_{a}yx^2v_{1}]+[v_{a}xyxv_{1}])-(3 \mu^{-1}+2)[v_{a+1}yx^2]-2[v_{a+1}xyx]\\
&=2(a-3)\mu^{-1}[v_{a+1}xyx]+3\mu^{-1}[v_{a+1}xyx]\\
&=(2a-3)\mu^{-1}[v_{a+1}xyx].
\end{align*}
%where we have used the additional relation $[v_{a}xyxv_{1}]=-3[v_{a+1}yx^2]-2[v_{a+1}xyx]=[v_{a+1}xyx]$.
Hence
$[v_{2}xy[v_{a-1}x]]=(-a+2)\mu^{-1}[v_{a+1}xyx]$,
and we find
\[
(a-1)\mu^{-1}[v_{a+1}xyx]=0,
\]
which implies $a\equiv 1\pmod{p}$ as desired.

If $\mu=0$ we have
\begin{align*}
&[v_{2}xyv_{a-1}x]=(a-3)[v_{a-2}xyv_{2}v_{1}x]-[v_{a-1}xyv_{2}x]\\
&=(a-3)[v_{a}yxv_{1}x]-2 [v_{a+1}yx^2]=(a-1)[v_{a+1}xyx],
\end{align*}
%where we have used the relation $[v_{a}yxv_{1}]=-2[v_{a+1}yx]-[v_{a+1}xy]$.
and
\begin{align*}
&[v_{2}xyxv_{a-1}]=(a-3)[v_{a-2}xyxv_{2}v_{1}]-[v_{a-1}xyxv_{2}]\\
&=2(a-3)[v_{a}yx^2v_{1}]-3[v_{a+1}yx^2]\\
&=2(a-3)[v_{a+1}xyx]+3[v_{a+1}xyx]=(2a-3)[v_{a+1}xyx],
\end{align*}
%because $[v_{a}yx^2v_{1}]=-3[v_{a+1}yx^2]-2[v_{a+1}xyx]=[v_{a+1}xyx]$.
Putting terms together we find
\[
(a-1)[v_{a+1}xyx]=0,
\]
whence $a\equiv 1\pmod{p}$ in this case as well.
\end{proof}

Thus, we have proved Assertion~(b) of Theorem~\ref{thm:main}.
Note, in particular, under our assumption
$L_{(a+1)(q-1)+3}\neq 0$,
our argument has also shown that if $a\equiv 1\pmod{p}$ then
$L_{(a+1)(q-1)+1}$ is a diamond of infinite type.
This will serve as basis of induction in one
argument of the next subsection.

\subsection{If $a \equiv 1 \pmod{p}$ then $a-1$ is a power of $p$.}\label{subsec:a-1_power_p}
Let $p^{s}$ be the highest power of $p$ dividing $a-1$.
Thus, $a=mp^{s}+1$ with $m$ not a multiple of $p$.
In this subsection we prove that $L$ has finite dimension
unless $m=1$.
Thus, assume $m>1$.
We start with showing that the diamond $L_{a(q-1)+1}$
of finite type $\mu$ is followed by
a number of diamonds of infinite type,
an inductive argument which we will partly
reuse later in Section~\ref{sec:fin_pres}.
Then we will use that to show finite-dimensionality
of $L$ in Proposition~\ref{prop:a=1+p^s},
which amounts to the more precise Assertion~(c) of
Theorem~\ref{thm:main}.

% In this subsection we prove that $L$ has at most $p^s-1$
% diamonds past $L_{a(q-1)+1}$, and hence
% it has finite dimension.
% We start with preliminary showing that all those $p^s-1$
% diamonds have type infinity,
% and then state and prove Proposition~\ref{prop:a=1+p^s},
% which amounts to the more precise Assertion~(c) of
% Theorem~\ref{thm:main}.

Define recursively the elements
\[
v_{a+k}=[v_{a+k-1}xyx^{q-3}]
\quad \textrm{for $1<k\le p^s$,}
\]
which extends the corresponding definition for $k=1$.
We prove by induction on $k$ that
$[v_{a+k}x]$ and $[v_{a+k}y]$ span a diamond of type infinity, for  $1\le k\le p^s$,
unless the next homogeneous component equals zero.
According to Theorem~\ref{thm:distance} we only need to
show
\[
[v_{a+k}yx]+[v_{a+k}xy]=0
\quad \textrm{for $1\le k\le p^s$.}
\]

% \begin{align*}
% &[v_{a+k}yx]+[v_{a+k}xy]=0,\\
% &[v_{a+k}xx]=0=[v_{a+k}yy], \\
% &[v_{a+k}yx^hy]=0  \quad 1\leq h \leq q-3,
% \end{align*}
% for $k=1, \ldots, p^s$.
% According to Theorem~\ref{thm:distance} only the first
% of the above relations needs a proof.

The case $k=1$ was proved in the previous subsection,
thus let $k>1$ and assume
the conclusion to hold up to $k-1$.
Because $a>p^{s}+1$ we have
$[v_{k+1}yx]=-[v_{k+1}xy]$ for $k\le p^s$.
Expanding each side of
\[
[v_{a-1}[v_{k+1}yx]]=-[v_{a-1}[v_{k+1}xy]]
\]
we obtain
\begin{equation}\label{eq:v_a+k}
[v_{a-1}v_{k+1}yx]+[v_{a-1}v_{k+1}xy]=2[v_{a-1}xv_{k+1}y]+2[v_{a-1}yv_{k+1}x].
\end{equation}
We distinguish two cases according to the value of $\mu$.
Suppose first $\mu \neq 0$.
Because $[v_{a-1}v_{k+1}]=-\mu^{-1}(k+1)v_{a+k}$,
$[v_{a-1}xv_{k+1}]=-\mu^{-1}[v_{a+k}x]$, and $[v_{a-1}yv_{k+1}]=-\mu^{-1}[v_{a+k}y]$, substituting
into Equation~\eqref{eq:v_a+k} we get
\[
(k-1)\bigl([v_{a+k}yx]+[v_{a+k}xy]\bigr)=0.
\]
We find the same conclusion when $\mu=0$,
because in that case a direct calculation shows
$[v_{a-1}v_{k+1}]=-(k+1)v_{a+k}$,
$[v_{a-1}xv_{k+1}]=-[v_{a+k}x]$, and $[v_{a-1}yv_{k+1}]=-[v_{a+k}y]$.
This proves the desired conclusion
$[v_{a+k}yx]+[v_{a+k}xy]=0$
as long as $k\not\equiv 1\pmod{p}$.

If $k \equiv 1 \pmod{p}$, write $k=hp^{t}+1$ where $p^t$ is the highest power of $p$ dividing
$k-1$, whence $h \not \equiv 0 \pmod p$.
Because $h\leq p^{s-t}-1$ we have $k+p^t\leq p^s+1$.
Since $a>p^s+1$ we have $[v_{k+p^t}yx]=-[v_{k+p^t}xy]$.
The equation
$[v_{a-p^{t}}[v_{k+p^t}yx]]=-[v_{a-p^{t}}[v_{k+p^t}xy]]$
yields
\begin{equation}\label{eq:k_cong_1}
[v_{a-p^t}v_{k+p^t}yx]+[v_{a-p^t}v_{k+p^t}xy]=2[v_{a-p^t}xv_{k+p^t}y]+2[v_{a-p^t}yv_{k+p^t}x].
\end{equation}
Once again we need to distinguish two cases
according to the value of $\mu$.
Assume first $\mu \neq 0$. We have
\begin{align*}
&[v_{a-p^t}v_{k+p^t}]=[v_{a-p^t}[v_{2}v_{1}^{k+p^t-2}]]\\
&=\binom{k+p^t-2}{p^t-1}[v_{a-1}v_{2}v_{1}^{k-1}]
-\binom{k+p^t-2}{p^t}[v_{a}v_{2}v_{1}^{k-2}]\\
&=-\mu^{-1}(2+h)v_{a+k},
\end{align*}
because
$
\binom{k+p^t-2}{p^t-1}=\binom{hp^{t}+p^{t}-1}{p^{t}-1}\equiv 1\pmod{p}$
and
$\binom{k+p^t-2}{p^t}=\binom{hp^{t}+p^{t}-1}{p^{t}}\equiv h\pmod{p}.$
% \[
% \binom{k+p^t-2}{p^t-1}=\binom{hp^{t}+p^{t}-1}{p^{t}-1}\equiv 1 \;\; \textrm{and}\;\;
% \binom{k+p^t-2}{p^t}=\binom{hp^{t}+p^{t}-1}{p^{t}}\equiv h.
% \]
Similarly, we have
\[
[v_{a-p^t}xv_{k+p^t}]=\binom{k+p^t-2}{p^t-1}[v_{a-1}xv_{2}v_{1}^{k-1}]
=-\mu^{-1}[v_{a+k}x],
\]
and
\[
[v_{a-p^t}yv_{k+p^t}]=\binom{k+p^t-2}{p^t-1}[v_{a-1}yv_{2}v_{1}^{k-1}]
=-\mu^{-1}[v_{a+k}y].
\]
Substituting into Equation~\eqref{eq:k_cong_1} we get
\[
h\bigl([v_{a+k}yx]+[v_{a+k}xy]\bigl)=0.
\]
We find the same conclusion when $\mu=0$ because
$[v_{a-p^t}v_{k+p^t}]=-(2+h)v_{a+k}$,
$[v_{a-p^t}xv_{k+p^t}]=-[v_{a+k}x]$ and $[v_{a-p^t}yv_{k+p^t}]=-[v_{a+k}y]$.

Thus, we have proved $[v_{a+k}yx]+[v_{a+k}xy]=0$
for $1\le k\le p^s$.
In particular, $[v_{a+p^s}yx]+[v_{a+p^s}xy]=0$ and
$[v_{a+p^s}yx]$ spans
$L_{(a+p^s)(q-1)+2}$.
In the next result, which amounts to Assertion~(c)
of~Theorem \ref{thm:main}, we show that
$[v_{a+p^s}yx]$
must actually vanish, under our present assumption $m>1$.

\begin{rem}
In case $m=1$, that is, when $a-1$ is a power of $p$,
one may still show that $[v_{a+k}x]$ and $[v_{a+k}y]$
span a diamond of infinite type for $1\le k< p^s$,
but the above argument would fail
because induction breaks down
whenever $k+p^t=p^s+1$.
However, in those cases one may use a different
calculation, as we will do
with a more general scope in the proof of Theorem~\ref{thm:fin_pres}.
\end{rem}

\begin{prop}\label{prop:a=1+p^s}
Under the hypotheses of Theorem~\ref{thm:main}, if $a\equiv 1 \pmod{p}$
but $a-1$ is not a power of $p$, then $L_{(a+p^s)(q-1)+2}=0$, where
$p^s$ is the highest power of $p$ which divides $a-1$.
\end{prop}

\begin{proof}
We have $a=mp^s+1$ with $a$ even, whence
$m$ is odd. Written $m=2h+1$, assume $h\neq 0$.
We will prove that $[v_{a+p^s}yx]=0$. Set $b=(h+1)p^s$ and consider the
element
\[
u=[v_{b}xyx^{(q-3)/2}]
\]
in degree $\bigl((a+p^s)(q-1)+2\bigr)/2$.
We expand
\begin{align*}
0&=[uu]=[u[v_{b}xyx^{(q-3)/2}]]=\sum_{i=0}^{(q-3)/2}(-1)^i\binom{(q-3)/2}{i}[ux^i[v_{b}xy]x^{(q-3)/2-i}]
\\&=
\sum_{i=0}^{(q-3)/2}(-1)^i\binom{(q-3)/2}{i}[ux^i[v_{b}x]yx^{(q-3)/2-i}]
\\&\quad
-
\sum_{i=0}^{(q-3)/2}(-1)^i\binom{(q-3)/2}{i}[ux^{i}y[v_{b}x]x^{(q-3)/2-i}].
\end{align*}
All terms of the former sum in the final difference vanish
except, possibly, for $i=(q-5)/2$ and $i=(q-3)/2$.
All terms of the latter sum vanish except, possibly,
for $i=(q-3)/2$.
Consequently, we find
\begin{align*}
0=\pm[uu]&=\frac{q-3}{2}[v_{b}xyx^{q-4}v_{b}xyx]-
\frac{q-3}{2}[v_{b+1}v_{b}yx]-[v_{b+1}v_{b}xy]+
\\&\quad
+[v_{b+1}xv_{b}y]+[v_{b+1}yv_{b}x]+[v_{b+1}xyv_{b}x].
\end{align*}

We now expand the Lie products in the above equation by means of Lemmas~\ref{lemma:v2} and~\ref{lemma:v2ext}.
We start with
\[
[v_{b+1}v_{b}]=[v_{b+1}[v_{2}v_{1}^{b-2}]]=\sum_{i=0}^{b-2}(-1)^{i}\binom{b-2}{i}[v_{b+1}v_{1}^iv_{2}v_{1}^{b-2-i}].
\]
The Lie product in the sum vanishes with the possible exceptions of when $b+1+i=a-1$, whence $i=b-p^{s}-1=hp^{s}-1$,
or $b+1+i=a$, whence $i=b-p^s=hp^{s}$.
In the former case the binomial coefficient vanishes modulo $p$ because
$
\binom{b-2}{b-p^s-1}=\binom{hp^s+p^s-2}{(h-1)p^s+p^s-1}
%\equiv 0 \pmod{p},
$,
and in the latter case the binomial coefficient is congruent to $1$ modulo $p$ because
$
\binom{b-2}{b-p^s}=\binom{hp^s+p^s-2}{hp^s}.
$
Hence
$
[v_{b+1}v_{b}]=(-1)^{hp^s}\mu^{-1}v_{a+p^s}
$
if $\mu \neq 0$, and
$
[v_{b+1}v_{b}]=(-1)^{hp^s}v_{a+p^s}
$
otherwise.

Next, we expand $[v_{b+1}xv_{b}]$ and obtain
\[
[v_{b+1}xv_{b}]=[v_{b+1}x[v_{2}v_{1}^{b-2}]]=\sum_{i=0}^{b-2}(-1)^{i}\binom{b-2}{i}[v_{b+1}xv_{1}^iv_{2}v_{1}^{b-2-i}].
\]
The Lie product in the sum vanishes, except possibly when $i=b-p^{s}-1$.
However, in that case the binomial coefficient vanishes, hence $[v_{b+1}xv_{b}]=0$, irrespectively of the value of $\mu \in \F$.
Similarly, we find $[v_{b+1}yv_{b}]=0$ and
$[v_{b+1}xyv_{b}]=0$.

Finally, we expand
\[
[v_{b}xyx^{q-4}v_{b}]=[v_{b}xyx^{q-4}[v_{2}v_{1}^{b-2}]]=(-1)^{hp^{s}}[v_{a-1}xyx^{q-4}v_{2}v_{1}^{p^{s}-2}].
\]
We obtain $(-1)^{hp^{s}}2\mu^{-1}[v_{a+p^s-1}xyx^{q-4}]$ if $\mu \neq 0$,
and $(-1)^{hp^{s}}2[v_{a+p^s-1}xyx^{q-4}]$ otherwise.
Here we have used the equation $[v_{a-1}xyx^{q-4}v_{2}]=2\mu^{-1}[v_{a+1}xyx^{q-4}]$ if $\mu \neq 0$,
and the equation $[v_{a-1}xyx^{q-4}v_{2}]=2[v_{a+1}xyx^{q-4}]$ otherwise.
Both can be verified by direct calculation.

Substituting all the Lie products found
into the equation $0=[uu]$
as expanded above, we conclude
$
[v_{a+p^s}yx]=0
$
as desired.
\end{proof}

\subsection{If $a\equiv 0 \pmod{p}$ then $a$ equals twice a power of $p$ }\label{subsec:a_cong_0}
In this final subsection we prove Assertion~(d)
of Theorem~\ref{thm:main}.
Thus, suppose the next diamond
of finite type past $L_q$ is
$L_{a(q-1)+1}$ with $a \equiv 0\pmod{p}$.
According to Assertion~(b) of Theorem~\ref{thm:main},
the algebra $L$ is then finite-dimensional unless
$\mu=1$ and $[L_{(a+1)(q-1)},y]=0$.
Here we show that, assuming those conditions as well,
one still concludes that $L$ has finite dimension unless
$a$ equals twice a power of $p$.
The proof follows a similar pattern as that of
Assertion~(c), and begins with showing that
the diamond $L_{a(q-1)+1}$,
here of type $\mu=1$, is followed by
a certain number of diamonds of infinite type.
However, the equation $[L_{(a+1)(q-1)},y]=0$
forces $L_{(a+1)(q-1)+2}$ to be the first such diamond,
rather than $L_{(a+1)(q-1)+1}$ as before.
This shift in degree will require adapting
notation and calculations used previously.

Thus, assume $a \equiv 0\pmod{p}$,
$\mu=1$, and $[L_{(a+1)(q-1)},y]=0$.
Taking into account that $a$ is even,
write $a=2mp^s$ with $m$ not a multiple of $p$.
Hence $p^s$ is the highest power of $p$ dividing $a$. We will prove $L_{(a+p^s)(q-1)+3}=0$ when $m >1$.

We redefine the element $v_{a+1}$ as
\[
v_{a+1}=[v_{a}xyx^{q-2}].
\]
Note that here $v_{a+1}$ spans $L_{(a+1)(q-1)+1}$,
rather than $L_{(a+1)(q-1)}$ as previously.
According to Theorem~\ref{thm:distance}
the component $L_{(a+1)(q-1)+2}$ is a diamond,
and in Subsection~\ref{subsec:a_not_equiv_1}
we have actually proved that it has infinite type,
hence $[v_{a+1}yx]=-[v_{a+1}xy]$.

Define recursively the elements
\[
v_{a+k}=[v_{a+k-1}xyx^{q-3}] \quad
\textrm{for $1<k\le p^s$}.
\]
We will prove by induction on $k$ that
$[v_{a+k}x]$ and $[v_{a+k}y]$ span a diamond of type infinity, for  $1\le k\le p^s$,
unless the next homogeneous component equals zero.
According to Theorem~\ref{thm:distance} we only need to
show
\begin{equation}\label{eq:v_a+k_new}
[v_{a+k}yx]+[v_{a+k}xy]=0
\quad \textrm{for $1\le k\le p^s$.}
\end{equation}

Before that we need to review our results in Section~\ref{sec:main} on the adjoint action of $v_{1}$
and $v_{2}$ on the elements close to the diamond $L_{a(q-1)+1}$, taking into account the updated definition of the element $v_{a+1}$.
As an immediate consequence of Lemmas~\ref{lemma:v1} and~\ref{lemma:v2ext}, we have
$[v_{a}xv_{1}]=v_{a+1}$,
$[v_{a}xyv_{1}]=-[v_{a+1}y]$, $[v_{a-1}xv_{2}]=-v_{a+1}$, $[v_{a-1}xyv_{2}]=[v_{a+1}y]$, and $[v_{a-1}xyxv_{2}]=2[v_{a+1}yx]=-2[v_{a+1}xy]$.

\begin{lemma}\label{lemma:v_a+1}
Under the above hypotheses we have
\begin{align*}
&[v_{a}xyxv_{1}]=[v_{a+1}xy],\\
&[v_{a}xv_{2}]=v_{a+2},\\
&[v_{a}xyv_{2}]=0\\
&[v_{a}xyxv_{2}]=[v_{a+2}yx]+[v_{a+2}xy].
\end{align*}
\end{lemma}

\begin{proof}
The second equation is
\begin{align*}
&[v_{a}xv_{2}]=[v_{a}x[v_{1}xyx^{q-3}]]=[v_{a}x[v_{1}xy]x^{q-3}]\\
&=[v_{a}x[v_{1}x]yx^{q-3}]=[v_{a+1}xyx^{q-3}]\\
&=v_{a+2}
\end{align*}
where we have used $[v_{a}xy[v_{1}x]]=[v_{a}xy[yx^{q-1}]]=[v_{a+1}yx]+[v_{a+1}xy]=0$
and $[v_{a}x[v_{1}x]]=[v_{a+1}x]$.
The remaining equations are similar.
\end{proof}

We now begin an inductive proof of
Equation~\eqref{eq:v_a+k_new}.
The case $k=1$ was known from previous subsections,
thus let $k>1$ and assume
the conclusion to hold up to $k-1$.
Note that $[v_{k}yx]+[v_{k}xy]=0$ for every $k$ in the range under consideration.
Expanding the equation
\[
[v_{a}x[v_{k}yx]]+[v_{a}x[v_{k}xy]]=0
\]
we obtain
$
[v_{a}xv_{k}yx]-2[v_{a}xyv_{k}x]+[v_{a}xv_{k}xy]+[v_{a}xyxv_{k}]=0.
$
Since $[v_{a}xv_{k}]=[v_{a}x[v_{2}v_{1}^{k-2}]]=[v_{a}xv_{2}v_{1}^{k-2}]=v_{a+k}$, $[v_{a}xyv_{k}]=0$ and
\[
[v_{a}xyxv_{k}]=(-1)^k[v_{a}xyxv_{1}^{k-2}v_{2}]
=(-1)^k([v_{a+k}yx]+[v_{a+k}xy])
\]
we conclude
\[
\bigl(1+(-1)^k\bigr)\bigl([v_{a+k}yx]+[v_{a+k}xy]\bigl)=0,
\]
which gives the desired conclusion as long as $k$ is even.

The case where $k$ is odd, as we assume from now on,
is harder and requires longer arguments.
We start with expanding
\[
[v_{a-1}x[v_{k+1}yx]]+[v_{a-1}x[v_{k+1}xy]]=0,
\]
thus obtaining
$
[v_{a-1}xv_{k+1}yx]-2[v_{a-1}xyv_{k+1}x]+[v_{a-1}xv_{k+1}xy]+[v_{a-1}xyxv_{k+1}]=0.
$
We have
\begin{align*}
[v_{a-1}xv_{k+1}]
&=[v_{a-1}x[v_{2}v_{1}^{k-1}]]
=[v_{a-1}xv_{2}v_{1}^{k-1}]-(k-1)[v_{a-1}xv_{1}v_{2}v_{1}^{k-2}]\\
&=-v_{a+k}-(k-1)v_{a+k}
=-kv_{a+k},
\end{align*}
and $[v_{a-1}xyv_{k+1}]=[v_{a-1}xyv_{2}v_{1}^{k-1}]=[v_{a+k}y]$. Furthermore, taking into account that $k$ is odd, we have
\begin{align*}
[v_{a-1}xyxv_{k+1}]
&=[v_{a-1}xyxv_{2}v_{1}^{k-1}]+[v_{a-1}xyxv_{1}^{k-1}v_{2}]\\
&=2\bigl(2[v_{a+k}yx]+[v_{a+k}xy]\bigr)-2\bigl([v_{a+k}yx]+[v_{a+k}xy]\bigr)\\
&=2[v_{a+k}yx].
\end{align*}
Hence we find
\[
k\bigl([v_{a+k}yx]+[v_{a+k}xy]\bigr)=0,
\]
which yields the desired conclusion as long as
$k$ is not a multiple of $p$.

If $k$ is a multiple of $p$
(and is odd as throughout),
write $k=hp^t$ with $h$ not a multiple of $p$.
Note that $[v_{k+p^t}yx]+[v_{k+p^t}xy]=0$.
In particular, when $k=p^s$ we have
$[v_{2p^s}yx]+[v_{2p^s}xy]=0$, since we are assuming $a>2p^s$.
Expanding
$
[v_{a-p^t}x[v_{k+p^t}yx]]+[v_{a-p^t}x[v_{k+p^t}xy]]=0
$
we find
\[
[v_{a-p^t}xv_{k+p^t}yx]-2[v_{a-p^t}xyv_{k+p^t}x]+[v_{a-p^t}xv_{k+p^t}xy]+[v_{a-p^t}xyxv_{k+p^t}]=0.
\]
Consider now the individual terms.
We have
\begin{align*}
[v_{a-p^t}xv_{k+p^t}]
&=[v_{a-p^t}x[v_{2}v_{1}^{k+p^t-2}]]\\
&=\sum_{i=0}^{k+p^t-2}(-1)^i \binom{k+p^t-2}{i}[v_{a-p^t}xv_{1}^{i}v_{2}v_{1}^{k+p^t-2-i}]\\
&=-h [v_{a-p^t}xv_{1}^{p^t}v_{2}v_{1}^{k-2}]\\
&=-hv_{a+k},
\end{align*}
where we have used
$
\binom{k+p^t-2}{p^t-1}=\binom{hp^t+p^t-2}{p^t-1}\equiv 0\pmod{p}
$
and
$
\binom{k+p^t-2}{p^t}\equiv h\pmod{p}.
$
%\[
%\binom{k+p^t-2}{p^t-1}=\binom{hp^t+p^t-2}{p^t-1}\equiv 0 \;\;
%\textrm{and} \;\;
%\binom{k+p^t-2}{p^t}\equiv h.
%\]
Similarly, we have
\[
[v_{a-p^t}xyv_{k+p^t}]=-[v_{a-p^t}xyv_{1}^{p^t-2}v_{2}v_{1}^{k}]=[v_{a+k}y],
\]
because $\binom{k+p^t-2}{p^t-2}\equiv 1\pmod{p}$.
Finally, considering that $k+p^t-2$ is even, we have
\begin{align*}
[v_{a-p^t}xyxv_{k+p^t}]
&=-[v_{a-p^t}xyxv_{1}^{p^t-2}v_{2}v_{1}^{k}]+[v_{a-p^t}xyxv_{1}^{k+p^t-2}v_{2}]
\\&=2[v_{a+k}yx].
\end{align*}
We conclude
$
h\bigl([v_{a+k}yx]+[v_{a+k}xy]\bigr)=0.
$

Thus, we have proved $[v_{a+k}yx]+[v_{a+k}xy]=0$ for $2\le k\le p^s$.
We are now ready to prove Assertion (d) of
Theorem~\ref{thm:main}, which we state again
as follows.

\begin{prop}
Under the hypotheses of Theorem~\ref{thm:main},
if $a\equiv 0\pmod{p}$, $\mu=1$, $[L_{(a+1)(q-1)},y]=0$, and $a/2$ is not a power of $p$,
then $L_{(a+p^s)(q-1)+3}=0$, where $p^s$ is the highest power of $p$ which divides $a$.
\end{prop}

\begin{proof}
Because $[v_{a+p^s}yx]+[v_{a+p^s}xy]=0$,
the desired conclusion
$L_{(a+p^s)(q-1)+3}=0$
now amounts to $[v_{a+p^s}yx]=0$.
That will follow if we show the stronger
equation $[v_{a+p^s}y]=0$.
We expand
\[
0=[v_{p^s}x[v_{a}y]]=[v_{p^s}xv_{a}y]-[v_{p^s}xyv_{a}].
\]
The former Lie product in this difference is
\[
[v_{p^s}xv_{a}y]=[v_{p^s}x[v_{2}v_{1}^{a-2}]y]=-[v_{p^s}xv_{1}^{a-p^s}v_{2}v_{1}^{p^s-2}y]=-[v_{a+p^s}y],
\]
where we have used
$
\binom{a-2}{a-1-p^s}
%=\binom{(2m-1)p^s+p^s-2}{(2m-2)p^s+p^s-1}
\equiv 0\pmod{p}
$
and
$
\binom{a-2}{a-p^s}
%=\binom{(2m-1)p^s+p^s-2}{(2m-1)p^s}
\equiv 1\pmod{p}.
$
The latter Lie product in the difference is
\[
[v_{p^s}xyv_{a}]=-(2m-1)[v_{p^s}xyv_{1}^{a-2-p^s}v_{2}v_{1}^{p^s}]=(2m-1)[v_{a+p^s}y],
\]
where we have used
$
\binom{a-2}{a-2-p^s}
%=\binom{(2m-1)p^s+p^s-2}{(2m-2)p^s+p^s-2}
\equiv 2m-1\pmod{p}.
$
We obtain
$
2m[v_{a+p^s}y]=0,
$
which yields the desired conclusion.
\end{proof}

\section{Nottingham algebras with diamonds of finite and infinite type}\label{sec:fin_pres}

Let $L$ be an infinite-dimensional Nottingham  algebra with second diamond
$L_q$ and standard generators $x$ and $y$. Suppose that $L$ has diamonds of infinite type in all
degrees $k(q-1)+1$ for $1<k\leq p^s$, where $s\geq 1$, and a diamond of finite type $\lambda\in \F$, with
$\lambda \neq 0$, in degree
$(p^s+1)(q-1)+1$.

In this section we prove that $L$ is uniquely determined by these prescriptions. It has diamonds in all degrees
of the form $t(q-1)+1$. If $t\not \equiv 1 \pmod {p^s}$ the corresponding diamond is of infinite type.
If $t\equiv 1 \pmod {p^s}$, say $t=rp^s+1$, the corresponding diamond has finite type $r\lambda +r-1$.
The diamonds of finite type
(including the fake ones if $\lambda\in \F_{p}$)
follow an arithmetic progression.
We prove this uniqueness result by showing that if a Lie algebra $N$ is defined
by a finite presentation encoding part of the above prescriptions (that is, up to specifying the type of the second
diamond of finite type), then the quotient $L$ of $N$ modulo its centre is a Nottingham algebra and has the structure
stated above.

% Nota: $q>5$ e' tecnicamente superfluo, ma c'e' nella nostra
% definizione di Nottingham

\begin{theorem}\label{thm:fin_pres}
Let $q>5$ be a power of a prime $p>3$, let $\F$ a field of characteristic $p$.
Fix $\lambda \in \F^{\ast}$, and a positive integer $s$.
Let $N=\bigoplus_{i=1}^{\infty}N_{i}$ be the Lie algebra over $\F$ on two generators $x$ and $y$
subject to the following relations, and graded assigning degree $1$ to $x$ and $y$,
where $v_{k}$ is defined recursively by $v_{1}=[yx^{q-2}]$ and
$v_{k}=[v_{k-1}xyx^{q-3}]$ for $k>1$:
\begin{align*}
&[yx^{i}y]=0 & &\textrm{for $0< i < q-2$,}\\
&[v_{1}xx]=0=[v_{1}yy],\quad
[v_{1}yx]=-2[v_{1}xy], &\\
&[v_{1}yx^iy]=0  & &\textrm{for $0< i < q-2$,}\\
&[v_{k}yx]+[v_{k}xy]=0 &  &\textrm{for $2\leq k \leq p^{s}$ with $k$ even,}\\
&\lambda [v_{p^s+1}yx]=(1-\lambda)[v_{p^s+1}xy]
\end{align*}
Then $N/Z(N)$ is a Nottingham algebra and has the diamond structure described above in the text.
\end{theorem}

Note in passing that the presentation in
Theorem~\ref{thm:fin_pres} does not include relations
$[v_{k}yx]+[v_{k}xy]=0$ for $k$ odd in the range $2<k\le p^s$,
because those are consequences of the remaining
relations, as we will see in its proof.

Naturally, we will prove Theorem~\ref{thm:fin_pres}
by induction, deducing homogeneous relations in each degree $j$
from those already proved in lower degree.
Hence, in essence, we will deduce relations in degree $j$
from already established properties of $N/N^j$.
In doing so we can make use of certain arguments
in~\cite{AviMat:diamond_distances}, as long as they
do not assume $N/N^{j+1}Z(N)$ or even larger quotients
of $N/Z(N)$ to be thin, which is something we actually need to
prove here.
Thus, we extract from~\cite{AviMat:diamond_distances}
a result, adapted to our present setting,
which follows from those arguments.
It is an adapted version
of~\cite[Theorem 10]{AviMat:diamond_distances}.

% For the proof of Theorem~\ref{thm:fin_pres}
% we will need a result from~\cite{AviMat:diamond_distances},
% which was originally stated
% for an infinite-dimensional Nottingham algebra $L$.
% We state them here in more general form
% which suits our present purposes.
% Their original proofs remain valid in this context.

\begin{theorem}
\label{thm:chain_recap}
Let $N=\bigoplus_{i=1}^{\infty}N_{i}$ be
a graded Lie algebra, generated by two elements $x$ and $y$ of $N_1$.

Suppose its quotient $M=N/N^{m+2}Z(N)$
is a Nottingham  algebra with second diamond $M_q$ and standard generators the images of $x$ and $y$,
% (which we denote by the same letters, abusing notation),
where $m\geq 2q-1$.

Suppose $M_{m}$ is a (possibly fake) diamond of $M$,
%of type $\mu \in \F\cup \{\infty\}$,
%for some $m\geq 2q-1$,
of type $\mu$.
If $\mu$ equals $-1$ or $0$, assume in addition that $M_{m-q+1}$ is a diamond with a type $\lambda$,
and in case $\mu=0$ assume
$\lambda\neq 0$.

% Suppose $y$ centralizes every homogeneous component from $L_{m-q+2}$ up to $L_{m-2}$.

Then $[N_iy]\subseteq Z(N)$ for $m< i\le m+q-3$.
%and actually $[N_iy]=0$ for $m< i\le m+q-3$ if $\mu\neq -1,2$.
% Then $y$ centralizes each homogeneous component from $N_{m+1}$ up to $N_{m+q-3}$.
\end{theorem}

% Part of the required calculations are the same as those
% employed in~\cite{AviMat:diamond_distances}
% $N/N^j$
% In particular, we wish to make use of our structural
% knowledge of Nottingham algebras, such as that summarized in
% Theorem~\ref{thm:distance},
% and apply it to the finite dimensional quotient $N/N^jZ(N)$.
% However, Theorem~\ref{thm:distance}, and other results
% in~\cite{AviMat:diamond_distances}, are stated
% for infinite-dimensional Nottingham algebras.
% Because certain calculations in their proofs
% still apply here, we extract
% from a result in~\cite{AviMat:diamond_distances}

% The above presentation shows that the quotient $N/N^{2q}$ is thin.

 %\myframe{Osservare che per la Prop.~\ref{prop:a_odd} risulta $[v_{k}yx]+[v_{k}xy]=0$ per ogni $2\leq k\leq p^s$. Il caso $k$
 %dispari \`e coperto dalla proposizione.}

\begin{proof}[Proof of Theorem~\ref{thm:fin_pres}]
Set $L=N/Z(N)$.
We will prove, inductively, that $L$ is thin,
and has the claimed structure.
More precisely, we will show that for all $k\ge 1$ we have
\begin{align}
&[L_{(k-1)(q-1)+1+i}y]=0  & &\textrm{for $0< i < q-2$,}
\label{eq:chain}
\\
&[v_{k}xx]=0=[v_{k}yy],
\label{eq:xxyy}
&
\\
&[v_{k}yx]+[v_{k}xy]=0 &  &\textrm{if
$k\not\equiv 1\pmod{p^s}$,}
\label{eq:infinite}
\\
&\mu_r[v_{rp^s+1}yx]=(1-\mu_r)[v_{rp^s+1}xy] &  &
\textrm{where $\mu_r=r(\lambda+1)-1$,}
\label{eq:finite}
% \\
% &(r\lambda+r-1)[v_{rp^s+1}yx]=(2-r\lambda-r)[v_{rp^s+1}xy] &  &
\\&
[v_{rp^s+1}y]=0  & &\textrm{if $\mu_r=1$,}
\label{eq:y}
\\&
[v_{rp^s+1}x]=0  & &\textrm{if $\mu_r=0$,}
\label{eq:x}
\end{align}
where $v_k$ denotes any nonzero element of $L_{k(q-1)}$.
We will naturally prove them by induction on the degree
of those equations.
There will be a main induction on $r$, then an induction on $0<k\le p^s$
to prove Equations~\eqref{eq:xxyy}
and~\eqref{eq:infinite} on the diamond types,
with Equations~\eqref{eq:finite}, \eqref{eq:y} and~\eqref{eq:x}
concerning the finite types,
%relating $[v_{rp^s+k}yx]$ and $[v_{rp^s+k}xy]$,
and for each $k$ an induction on $i$ to prove Equation~\eqref{eq:chain}.

The presentation of $N$ tells us explicitly that the quotient
$N/N^{2q}$ is thin (note $N_{2q}=[v_2N_1N_1]$),
and thus a Nottingham algebra, with second diamond in degree $q$.
We start our proof with showing that
$L/L^{(p^s+1)(q-1)+3}$ is thin.
(It is actually the case that $N/N^{(p^s+1)(q-1)+3}$
itself is thin, but we do not need that here.)
Note $L_{(p^s+1)(q-1)+3}=[v_{p^s+1}L_1L_1L_1]$.
Thus, we show by induction that, in $L$, we have
\begin{align*}
&[v_{k-1}yx^iy]=0  &&\textrm{for $0< i < q-2$,}\\
&[v_kxx]=0=[v_kyy],
\end{align*}
for $2\le k\le p^s+1$.
For $k=2$ the first set of relations
is included in the presentation of $N$.
Because
\[
0=[v_1xyx^{q-4}[xyy]]=[v_2yy]
\]
and
\[
0=[v_1[v_1xx]]=[v_1[yx^q]]=[v_1yx^q]-[v_1x^qy]=[v_1yx^q]=[v_2xx],
\]
we have $[v_2yy]=0$ and $[v_2xx]=0$.
Hence $N/N^{2q+1}$ is thin, and has a a diamond in degree
$2q-1$, of infinite type as imposed
by the presentation of $N$.
Now let $2<k\le p^s+1$ and assume the conclusions hold
for all smaller values of $k$.
According to Theorem \ref{thm:chain_recap}
with $m=(k-1)(q-1)+1$, in $L$ we have
\[
[v_{k-1}yx^iy]=0\qquad\textrm{for $0< i < q-2$.}
\]
As before, because
\[
0=[v_{k-1}xyx^{q-4}[xyy]]=[v_kyy]
\]
and
\[
0=[v_{k-1}[v_1xx]]=[v_{k-1}[yx^q]]=[v_{k-1}yx^q]-[v_{k-1}x^qy]=[v_{k-1}yx^q]=[v_kxx]
\]
we find $[v_kyy]=0$,
and $[v_kxx]=0$.
Thus, $L/L^{k(q-1)+3}$ is thin, and has a diamond
in degree $k(q-1)+1$.
That diamond has infinite type if $k\le p^s$,
as imposed by the presentation of $N$ for even $k$,
and because of Proposition~\ref{prop:a_odd} for odd $k$,
and type $\lambda$ if $k=p^s+1$.
This completes our induction, and hence
$L/L^{(p^s+1)(q-1)+3}$ is thin, with the diamond
structure announced.

In the rest of the proof we will proceed inductively,
by successive spans of $p^s(q-1)$ in degree,
to prove $L$ is thin and has the diamond structure announced.
In particular, letting $a_r=rp^s+1$ for $r\geq 1$,
we will proceed by induction on $r$,
assuming
$\mu_r [v_{a_r}yx]=(1-\mu_r)[v_{a_r}xy]$ where $\mu_r=r(\lambda+1)-1$,
hence $[v_{a_r} x]$ and $[v_{a_r} y]$ span a diamond of type $\mu_r$,
preceded by  a string of $p^s-1$ diamonds of infinite type,
in the appropriate degrees.
Note that what we have proved so far
constitutes a virtual case $r=0$.
% As stipulated above, $v_{a}$ is a nonzero element in degree $a(q-1)$,
% and by induction we may assume
% $\mu_r [v_{a}yx]=(1-\mu_r)[v_{a}xy]$ where $\mu_r=r\lambda +r-1$.
To simplify notation we will denote $a_r$ by $a$
(with a warning that its meaning is more general than how it was used
in Section~\ref{sec:proof_main}, the proof of Theorem~\ref{thm:main}).

In our description of $L$ at the beginning
of this proof we have called $v_k$ any
nonzero element of $L_{k(q-1)}$,
hence defined only up to a scalar multiple.
Now we refine this by making a definite choice of scalar.
Thus, define the element $v_{a+1}$ in degree $(a+1)(q-1)$ as
\[
v_{a+1}=\begin{cases}
[v_{a}xyx^{q-3}] &\mbox{if } \mu_r \neq 0, \\
[v_{a}yx^{q-2}] &\mbox{otherwise},
\end{cases}
\]
and, recursively, the elements
\[
v_{a+k}=[v_{a+k-1}xyx^{q-3}] \quad \textrm{for $1<k\le p^s$.}
\]

We now prove that every homogeneous component of degree $(a+k)(q-1)+1$
in $L$ is a diamond of infinite type,
for $1\le k<p^{s}$, and the homogeneous component of degree
$(a+p^s)(q-1)+1$ is a diamond of (finite) type
$\mu_{r+1}=r\lambda+r+\lambda$.
%$\bar{\mu}=\mu+\lambda+1=(r+1)\lambda+r$.
Therefore, we prove
\begin{align*}
&[L_{(a+k-1)(q-1)+1+i}y]=0
&&\textrm{for $1\leq k \leq p^s$ and $0<i<q-2$,}
\\
&[v_{a+k}xx]=0=[v_{a+k}yy] & &\textrm{for $1\leq k \leq p^s$,}  \\
&[v_{a+k}yx]+[v_{a+k}xy]=0 & &\textrm{for $1\leq k \leq p^s-1$,} \\
&\mu_{r+1}[v_{a+p^s}yx]=(1-\mu_{r+1})[v_{a+p^s}xy], & &%\textrm{where $\bar{\mu}=\mu+\lambda+1=(r+1)\lambda+r$.}
\end{align*}
in $L$.
In our overall induction on $r$ we will prove
the induction base at the same time as the induction step,
relying on a virtual case $r=0$ which
essentially amounts to what we have proved so far.

We now proceed by nested inductions on $k$ and $i$.
Thus, let $1\le k\le p^s$ and assume the conclusions hold
for all smaller values of $k$.
(This means none when $k=1$, where, however,
the required information is available from the previous part of the proof.)
According to Theorem \ref{thm:chain_recap}
with $m=(a+k-1)(q-1)+1$, in $L$ we have
\[
[L_{(a+k-1)(q-1)+1+i}y]=0\qquad\textrm{for $0< i < q-2$.}
\]
Letting $u$ a homogeneous element with $[ux]=v_{a+k}$,
because $[uy]=0$ we obtain
\[
0=[u[xyy]]=[v_{a+k}yy].
\]
Unless $k=1$ and $\mu_r=1$, we obtain
\[
0=[v_{{a+k}-1}[yx^q]]
=[v_{{a+k}-1}yx^q]-[v_{{a+k}-1}x^qy]
%=[v_{{a+k}-1}yx^q]
=[v_{a+k}xx].
\]
We will deal with the excluded case below,
where $[v_{a+1}xx]$ will turn out to only be central in $N$.

Now we complete the case $k=1$ of our induction by proving $[v_{a+1}yx]+[v_{a+1}xy]=0$ in $L$.
If $\mu_r \neq 1$ we may proceed as in Subsection~\ref{subsec:a_not_equiv_1}.
Thus, the equation $[v_{a-1}x[v_{2}yx]]+[v_{a-1}x[v_{2}xy]]=0$ yields
$[v_{a+1}yx^2]+[v_{a+1}xyx]=0$.
Consequently, the element $[v_{a+1}yx]+[v_{a+1}xy]$ is central
in $N$ and the conclusion follows.
If $\mu_r=1$, then the equation $[v_{a-1}x[v_{2}yx]]+[v_{a-1}x[v_{2}xy]]=0$ yields
$[v_{a+1}xxy]+[v_{a+1}yx^2]+[v_{a+1}xyx]=0$.
Furthermore, we obtain $0=[v_{a}x[yx^q]]=[v_{a+1}x^3]$.
Consequently, once we prove $[v_{a+1}yx]+[v_{a+1}xy]=0$
it will also follow that the element $[v_{a+1}x^2]$ is central in $N$,
as we announced above.
Let $v_{a-p^s}$ be a nontrivial element in degree $(a-p^s)(q-1)$, then $v_{a-p^s}$ is just above a
diamond of (finite) type $\mu_{r-1}=\mu_r-(\lambda+1)=-\lambda$.
We expand both terms
of the equation
\begin{equation}\label{eq:v_a-p^s}
\lambda[v_{a-p^s}[v_{p^s+1}yx]]=(1-\lambda)[v_{a-p^s}[v_{p^s+1}xy]],
\end{equation}
by means of Lemmas~\ref{lemma:v2} and~\ref{lemma:v2ext}.
The Lie bracket on the left-hand side is
\begin{align*}
[v_{a-p^s}[v_{p^s+1}yx]]
&=
[v_{a-p^s}v_{p^s+1}yx]-[v_{a-p^s}yv_{p^s+1}x]
\\&\quad
-[v_{a-p^s}xv_{p^s+1}y]+[v_{a-p^s}xyv_{p^s+1}]
\\&=
\lambda^{-1}[v_{a+1}yx]+[v_{a+1}xy]
\end{align*}
because $[v_{a-p^s}v_{p^s+1}]=-2v_{a+1}$,
$[v_{a-p^s}yv_{p^s+1}]=-(\lambda^{-1}+1)[v_{a+1}y]$,
$[v_{a-p^s}xv_{p^s+1}]=-[v_{a+1}x]$
and $[v_{a-p^s}xyv_{p^s+1}]=[v_{a+1}yx]$.
The Lie product at the right-hand side is
\begin{align*}
[v_{a-p^s}[v_{p^s+1}xy]]
&=
[v_{a-p^s}v_{p^s+1}xy]-[v_{a-p^s}xv_{p^s+1}y]
\\&\quad
-[v_{a-p^s}yv_{p^s+1}x]+[v_{a-p^s}yxv_{p^s+1}]
-[v_{a+1}xy],
\end{align*}
where we have used $[v_{a-p^s}yx]=-(\lambda^{-1}+1)[v_{a-p^s}xy]$ due to this diamond's type.
Substituting in Equation~\eqref{eq:v_a-p^s} we get $[v_{a+1}yx]=-[v_{a+1}xy]$, as desired.

We continue our induction on $k$ and prove
$[v_{a+k}yx]+[v_{a+k}xy]=0$ in $L$
for $1<k\le p^s-1$.
Because $[v_{k+1}yx]+[v_{k+1}xy]=0$, expanding the equation $[v_{a-1}[v_{k+1}yx]]+[v_{a-1}[v_{k+1}xy]]=0$
as in Subsection~\ref{subsec:a-1_power_p} yields $[v_{a+k}yx]+[v_{a+k}xy]=0$ when $k \not \equiv 1 \pmod{p}$.
The argument used in Subsection~\ref{subsec:a-1_power_p} to cover the case $k \equiv 1 \pmod{p}$ works in the
present setting as well, unless $k+p^t=p^s+1$ for some $t$ with $1\leq t \leq s-1$.
If $k=p^s-p^t+1$ then we use the equation
\begin{equation}\label{eq:v_a+k_infty}
[v_{a-1}x[v_{k+1}yx]]+[v_{a-1}x[v_{k+1}xy]]=0.
\end{equation}
The former term in the sum is $[v_{a-1}x[v_{k+1}yx]]=[v_{a-1}xv_{k+1}yx]-[v_{a-1}xyv_{k+1}x]$.
The latter term is
$[v_{a-1}x[v_{k+1}xy]]=[v_{a-1}xv_{k+1}xy]-[v_{a-1}xyv_{k+1}x]+[v_{a-1}xyxv_{k+1}]$.
Assume first $\mu_r \neq 0$, and recall $[v_{a-1}xv_{k+1}]=-\mu_{r}^{-1}[v_{a+k}x]$
from Subsection~\ref{subsec:a-1_power_p}.
Taking into account that $k$ is odd we find
\begin{align*}
[v_{a-1}xyv_{k+1}]
&=
[v_{a-1}xy[v_{2}v_{1}^{k-1}]]=[v_{a-1}xyv_{2}v_{1}^{k-1}]+[v_{a-1}xyv_{1}^{k-1}v_{2}]
\\&=
-2\mu_{r}^{-1}[v_{a+k-1}xyv_{1}]-(2\mu_{r}^{-1}-1)[v_{a+k-2}xyv_{2}]
\\&=
(2\mu_{r}^{-1}+1)[v_{a+k}yx]+[v_{a+k}xy],
\end{align*}
and we deduce
$[v_{a-1}x[v_{k+1}yx]]=-(2\mu_{r}^{-1}+1)[v_{a+k}yx^{2}]-(\mu_{r}^{-1}+1)[v_{a+k}xyx]$.
Since
\begin{align*}
[v_{a-1}xyxv_{k+1}]
&=
[v_{a-1}xyx[v_{2}v_{1}^{k-1}]]=[v_{a-1}xyxv_{2}v_{1}^{k-1}]+[v_{a-1}xyxv_{1}^{k-1}v_{2}]
\\&=
-3\mu_{r}^{-1}[v_{a+k-1}xyxv_{1}]-(3\mu_{r}^{-1}-1)[v_{a+k-2}xyv_{2}]
\\&=
(3\mu_{r}^{-1}+2)[v_{a+k}yx^{2}]+2[v_{a+k}xyx],
\end{align*}
we have $[v_{a-1}x[v_{k+1}xy]]=(\mu_{r}^{-1}+1)[v_{a+k}yx^{2}]+[v_{a+k}xyx]$.
We conclude
\[
\mu_{r}^{-1}\bigl([v_{a+k}yx^{2}]+[v_{a+k}xyx]\bigr)=0,
\]
and so the element $[v_{a+k}yx]+[v_{a+k}xy]$ is central in $N$.
In the excluded case $\mu_{r}=0$, Equation~\eqref{eq:v_a+k_infty} yields $[v_{a+k}yx^{2}]+[v_{a+k}xyx]=0$,
because $[v_{a-1}xv_{k+1}]=-[v_{a+k}x]$,
$[v_{a-1}xyv_{k+1}]=2[v_{a+k}yx]$,
and $[v_{a-1}xyxv_{k+1}]=3[v_{a+k}yx^2]$.

Thus, we have proved that each homogeneous component of $L$
of degree $(a+k)(q-1)+1$ is a diamond of infinite type, for $1\le k\le p^{s}-1$.
To complete the proof we show
\[
\mu_{r+1}[v_{a+p^s}yx]=(1-\mu_{r+1})[v_{a+p^s}xy],
\]
in $N$.
To this purpose we expand both sides of the equation
\begin{equation}\label{eq:v_a+p^s}
\lambda [v_{a-1}[v_{p^s+1}yx]]=(1-\lambda)[v_{a-1}[v_{p^s+1}xy]],
\end{equation}
and assume first $\mu_r \neq 0$.
When $k=p^s$ the calculations we did in the previous paragraph for $k$ odd yield $[v_{a-1}xv_{p^s+1}]=-\mu_{r}^{-1}[v_{a+p^s}x]$
and $[v_{a-1}xyv_{p^s+1}]=(2\mu_{r}^{-1}+1)[v_{a+p^s}yx]+[v_{a+p^s}xy]$.
Recalling the equations $[v_{a-1}v_{p^s+1}]=-\mu_{r}^{-1}v_{a+p^s}$ and
$[v_{a-1}yv_{p^s+1}]=-\mu_{r}^{-1}[v_{a+p^s}y]$
from Subsection~\ref{subsec:a-1_power_p}, we find
\[
\lambda\mu_{r} [v_{a-1}[v_{p^s+1}yx]]=\lambda(2+\mu_{r})[v_{a+p^s}yx]+\lambda(1+\mu_{r})[v_{a+p^s}xy],
\]
and
\[
(1-\lambda)\mu_{r} [v_{a-1}[v_{p^s+1}xy]]=(\lambda-1)(1+\mu_{r})[v_{a+p^s}yx]+(\lambda-1)\mu_{r}[v_{a+p^s}xy].
\]
Substituting in Equation~\eqref{eq:v_a+p^s} and multiplying both sides by $\mu_r$ we obtain
\[
(\mu_{r}+\lambda+1)[v_{a+p^s}yx]=-(\mu_{r}+\lambda)[v_{a+p^s}xy],
\]
as desired.
Now we expand Equation~\eqref{eq:v_a+p^s}
assuming $\mu_{r}=0$, whence $\mu_{r+1}=\lambda+1$.
 %\[
 %\lambda[v_{a-1}[v_{p^s+1}yx]]=(1-\lambda)[v_{a-1}[v_{p^s+1}xy]].
 %\]
The above calculations for $k$ odd (and $\mu_{r}=0$) give $[v_{a-1}xv_{p^s+1}]=-[v_{a+p^s}x]$ and
$[v_{a-1}xyv_{p^s+1}]=2[v_{a+p^s}yx]$ when $k=p^s$.
Recalling the equations $[v_{a-1}v_{p^s+1}]=-v_{a+p^s}$ and
$[v_{a-1}yv_{p^s+1}]=-[v_{a+p^s}y]$
 from Subsection~\ref{subsec:a-1_power_p}, we find
\[
\lambda[v_{a-1}[v_{p^s+1}yx]]=2\lambda[v_{a+p^s}yx]+\lambda[v_{a+p^s}xy]
\]
and
\[
(1-\lambda)[v_{a-1}[v_{p^s+1}xy]]=(\lambda-1)[v_{a+p^s}yx].
\]
We obtain $(\lambda+1)[v_{a+p^s}yx]=-\lambda[v_{a+p^s}xy]$, thus concluding the proof.
\end{proof}

\begin{rem}\label{rem:lambda=0}
The conclusion of Theorem~\ref{thm:fin_pres} extends to the case $\lambda=0$ excluded once the additional relation
$[v_{2p^s+2}xx]=0$ is included, at the expense of some additional calculations, which we now outline.
The proof of Theorem~\ref{thm:fin_pres} only fails to show that if $v_{a}$ is just above a diamond of type
$\mu_{r}=1$ then $v_{a+1}$ lies just above a diamond of infinite type.
As in the general case, one has
$[v_{a+1}xxy]+[v_{a+1}yx^2]+[v_{a+1}xyx]=0=[v_{a+1}x^3]$.
However, expanding Equation~\eqref{eq:v_a-p^s} is inconclusive when $\lambda=0$.
To get around this, at the first occurrence of a diamond of type $\mu_{r}=1$, which occurs for $a=2p^s+1$,
the additional relation $[v_{2p^s+2}xx]=0$ allows one to conclude.
For $a>2p^s+1$ one uses the additional relation to expand $[v_{a-2p^s-1}y[v_{2p^s+2}xx]]=0$.
\end{rem}

\bibliography{References}

\end{document}